\documentclass[leqno,12pt]{amsart}

\title[Graded   contractions of the exceptional Lie algebras]{Graded   contractions of the $\mathbb{Z}_2^3$-gradings on the \\ exceptional Lie algebras coming from octonions}

\author[]{Francisco Cuenca, Cristina Draper$^{\star}$ and Thomas L. Meyer 
}

 \newcommand{\Addresses}{{
  \bigskip
  \footnotesize

  F.~Cuenca Carr\'egalo, \textsc{Universidad de M\'alaga (Spain)}\par\nopagebreak
  \textit{E-mail address}: \texttt{fccsjpq@uma.es}

  \medskip

  C.~Draper (Corresponding author), 
  \textsc{ Universidad de M\'alaga (Spain)}\par\nopagebreak
  \textit{E-mail address}: \texttt{cdf@uma.es}
  
   \medskip
  
   T.~L.~Meyer, \textsc{University of Cape Town (South Africa)}\par\nopagebreak
  \textit{E-mail address}: \texttt{tomleenenmeyer@gmail.com}

}}

\usepackage{amssymb,color, bm}
\usepackage{stackrel}


\usepackage{amssymb, mathrsfs, amsmath, tabu, amsthm}
\usepackage{hyperref}
\usepackage{tikz-cd, graphicx} 
\usepackage{pdfpages}
\usepackage{enumitem}
\usepackage{ragged2e}
\usetikzlibrary{decorations.markings}
\usepackage{tabularx}
\usepackage{longtable}
\usepackage{diagbox}

\hypersetup{
	colorlinks = true,
	urlcolor = blue,
	linkcolor = blue,
	citecolor = red
}

\newcommand{\mf}[1]{\mathfrak{#1}}
\newcommand{\f}[1]{\mathfrak{#1}}
 
\newcommand{\ep}{\varepsilon}

\newcommand{\weyl}{\mathcal{W}}

\newcommand{\comment}[1]{}
\newcommand{\mcal}[1]{\mathcal{#1}}

\DeclareMathOperator{\Aut}{Aut}
\DeclareMathOperator{\Stab}{Stab}

\DeclareMathOperator{\id}{id}

\theoremstyle{plain} 
\newtheorem{theorem}{Theorem}[section]
\newtheorem{lemma}[theorem]{Lemma}
\newtheorem{coro}[theorem]{Corollary}
\newtheorem{prop}[theorem]{Proposition}

\theoremstyle{definition} 
\newtheorem{defi}[theorem]{Definition}

\newtheorem{remark}[theorem]{Remark}
\newtheorem{example}[theorem]{Example}

\begin{document}

\maketitle

\begin{abstract}
A total of 860 nonisomorphic \( \mathbb{Z}_2^3 \)-graded Lie algebras of dimensions 52, 78, 133, and 248 are obtained as graded contractions of the \( \mathbb{Z}_2^3 \)-gradings on the exceptional Lie algebras (excluding \( \mathfrak{g}_2 \)) arising from the octonions.  
It is shown that all graded contractions of these gradings are necessarily generic. Their supports correspond to a combinatorial object known as a \emph{generalised nice set}, which serves as the main tool in the classification.  
The resulting algebras are distinguished by their Levi decompositions, the derived series of their radicals, their centers, and other structural features.
\end{abstract}

\section{Introduction}  

Gradings have been a key tool in the study of Lie algebras ever since Killing classified complex simple Lie algebras using the grading over the root system---the so-called Cartan grading, which is a grading by a free abelian group whose rank equals that of the Lie algebra. Gradings also appear naturally in applications of Lie algebras to various contexts, particularly in Geometry and Physics.
Patera and Zassenhaus initiated a systematic development of the theory of Lie algebra gradings in \cite{patera1998lie_gradings}, a program that culminated in the AMS monograph \cite{elduque2013gradings_simple_lie} by Elduque and Kochetov, devoted to gradings on simple Lie algebras. This work compiles contributions from many authors concerning classification results up to equivalence and up to isomorphism. An accessible survey on the classical case can be found in \cite{survey}, and another one on the exceptional case is available in \cite{mioE}. The latter emphasizes concrete descriptions of gradings, as some theoretical classifications of gradings on exceptional Lie algebras are still not completely known.

 Exceptional Lie algebras are particularly difficult to work with, and gradings often offer novel approaches for their study. Gradings on \( \mathfrak{g}_2 \) were independently classified in \cite{draper-martin2006linear,otroG2}. The unique non-toral \( \mathbb{Z}_2^3 \)-grading on \( \mathfrak{g}_2 \) made it possible to construct convenient bases adapted to explicit models of both the complex algebra \( \mathfrak{g}_2 \) and its compact real form, \cite{draper2024twistedg2}.
The remaining exceptional Lie algebras are even more challenging to handle. Fortunately, Tits provided a unified model for all of them in \cite{tits1966algebres}, which remains impressive to this day for its elegance and simplicity. This construction makes it possible to consider \( \mathbb{Z}_2^3 \)-gradings simultaneously on all complex exceptional Lie algebras other than \( \mathfrak{g}_2 \). These gradings share many symmetry properties, and our goal in this work is to investigate this family of exceptional Lie algebras through the lens of their \( \mathbb{Z}_2^3 \)-gradings, more specifically, by analyzing their graded contractions.

The study of contractions of Lie algebras began in physics  with the foundational work \cite{InonuWigner53}, where standard contractions were used to relate the symmetry algebras of different physical theories---for example, connecting relativistic and classical mechanics. While graded contractions are a different notion, they were later explored, partly inspired by these ideas, and gained interest among physicists as well (see \cite{patera1991discrete}). Their study provides new algebraic perspectives and can reveal structural relationships between Lie algebras beyond traditional contractions.
Contraction procedures provide useful relations between Lie algebras that can assist in their classification. 
Some general notions and invariants were introduced in \cite{Weimar-WoodsI}, which aimed to characterize \( G \)-graded contractions for arbitrary abelian groups, without reference to a particular Lie algebra. 
A recent contribution in this direction is \cite{kochetov2025generic}, which develops a general theory of generic graded contractions for arbitrary gradings, focusing on structural properties and families of contractions rather than on concrete classifications.
In contrast, a more specific approach is adopted in \cite{draper2024gradedg2}, where the Lie algebra under consideration is the particular case of \( \mathfrak{g}_2 \). Unlike most previous works following this \lq\lq specific\rq\rq approach---in which most homogeneous components have dimension one---the grading in \cite{draper2024gradedg2} features components all of dimension two.
Moreover, a complete classification of the algebras obtained by graded contraction was achieved in \cite{draper2024gradedg2}, notably without the use of a computer. The key to this achievement was the introduction of a combinatorial object, the so-called \emph{nice set}.
The nice sets turned out to be precisely the supports of all graded contractions of the \( \mathbb{Z}_2^3 \)-grading on \( \mathfrak{g}_2 \), enabling a complete classifications under different equivalence relations---yielding isomorphic and gr-isomorphic   Lie algebras. Furthermore, nice sets were instrumental in the classification of graded contractions of the complex orthogonal Lie algebras \( \mathfrak{so}(7) \) and \( \mathfrak{so}(8) \) in \cite{draper2024gradedb3d4}. 
The connection between these two studies illustrates that two Lie algebras graded by the same group can, in some cases, give rise to exactly the same graded contraction classifications---even though the resulting Lie algebras are, of course, not the same.
This is partly the case in the present work, but not entirely. Here, we study the graded contractions of the exceptional Lie algebras of dimension at least 52, all endowed with the aforementioned \( \mathbb{Z}_2^3 \)-grading arising from octonions. The resulting classification is entirely different from that of \( \mathfrak{g}_2 \), yet it can still be treated in a unified way. To this end, we introduced and classified in \cite{draper2024generalised} a new combinatorial object called the \emph{generalised nice set}, which serves as the key tool in our approach.  
All supports of generic graded contractions turn out to be generalised nice sets, and all graded contractions of the \( \mathbb{Z}_2^3 \)-gradings considered in this work are, in fact, generic. Remarkably, there exists a large number of such generalised nice sets---most of them corresponding to non-isomorphic Lie algebras.
  \medskip


The paper is structured as follows. In Section~\ref{se_prelim}, we recall some background on gradings and graded contractions of Lie algebras. Next, we explain how the notion of a nice set served as the key tool for the classification of the graded contractions of three 
$\mathbb{Z}_2^3$-gradings on the Lie algebras $\mathfrak{g}_2$, $\mathfrak{b}_3$, and $\mathfrak{d}_4$.
In Section~\ref{se_GNS}, we introduce a new combinatorial object that plays the role of the nice set in our context: the so-called generalised nice set. Its classification is briefly but clearly summarised in Theorem~\ref{teo_listaGNS}. It is worth noting that not every generalised nice set is a nice set, and the converse also fails. 
Section~\ref{sec_principal} deals with the $\mathbb{Z}_2^3$-gradings on the exceptional Lie algebras $\mathfrak{f}_4$, $\mathfrak{e}_6$, $\mathfrak{e}_7$, and $\mathfrak{e}_8$. These can be described simultaneously using Tits' unified construction of the exceptional Lie algebras, by varying the Jordan algebra involved in the construction.
The key difference between their graded contractions and those of \( \mathfrak{g}_2 \), \( \mathfrak{b}_3 \), and \( \mathfrak{d}_4 \) is addressed in Section~\ref{se_generica}, where it is shown that any graded contraction in our setting is necessarily generic and, in particular, has support given by a generalised nice set. 
The abundance of generalised nice sets, together with the elementary Proposition~\ref{prop_ep^T}, provides a large collection of Lie algebras that can be obtained as graded contractions of the four \( \mathbb{Z}_2^3 \)-gradings.
Section~\ref{se_propiedades} is devoted to studying the properties of these Lie algebras, which can, quite remarkably, be described with almost no reference to the underlying Hurwitz algebra.  
We observe that the algebras in Freudenthal's magic square appear precisely as the Levi factors of the algebras obtained by graded contraction.
Finally, in Theorem~\ref{teo_representantes}, we study how many non-isomorphic Lie algebras can be obtained in this way (up to graded isomorphism), arriving at 860 algebras---215 of each of the dimensions 52, 78, 133, and 248. As before, the problem of distinguishing these algebras is closely related to combinatorics.

Throughout this paper, we work over an algebraically closed field \( \mathbb{F} \) of characteristic zero.


\section{Preliminaries }\label{se_prelim}
	
\subsection{Gradings and graded contractions of Lie algebras}\label{se_gradings}

 Given an abelian group \( G \), a \( G \)-grading \( \Gamma \) on an algebra \( \mathcal{A} \) is a vector space decomposition \( \Gamma: \mathcal{A} = \bigoplus_{g\in G} \mathcal{A}_g \) such that \( \mathcal{A}_g \mathcal{A}_h \subseteq \mathcal{A}_{g + h} \) for all \( g, h \in G \). Each subspace \( \mathcal{A}_g \) is called a \emph{homogeneous component}, and \( g \) its \emph{degree}. A subspace \( W \subseteq \mathcal{A} \) is called a \emph{homogeneous subspace} if \( W = \bigoplus_{g \in G} (W \cap \mathcal{A}_g) \). 
 
 For example, if \( V = \bigoplus_{g \in G} V_g \) is a vector space decomposition, then the general linear Lie algebra \( \mathfrak{gl}(V) = \bigoplus_{g \in G} \mathfrak{gl}(V)_g \) is \( G \)-graded, where the homogeneous components are given by \( \mathfrak{gl}(V)_g = \{ f \in \mathfrak{gl}(V) \mid f(V_h) \subseteq V_{g+h} \ \forall h \in G \} \). Moreover, if we start with a \( G \)-graded algebra \( \mathcal{A} = \bigoplus_{g \in G} \mathcal{A}_g \), then the derivation algebra \( \mathfrak{der}(\mathcal{A}) = \{ d \in \mathfrak{gl}(\mathcal{A}) \mid d(xy) = d(x)y + x d(y) \ \forall x,y \in \mathcal{A} \} \), which is a Lie subalgebra of \( \mathfrak{gl}(\mathcal{A}) \), is also a homogeneous subspace; in particular, \( \mathfrak{der}(\mathcal{A}) \) is \( G \)-graded as well. We will apply this to different algebras in Section~\ref{se_TitsConstruction}, such as composition algebras and Jordan algebras.

We can consider several groups associated with a grading \( \Gamma \) on an algebra \( \mathcal{A} \). First, the automorphism group of the grading:
\[
\Aut(\Gamma) = \left\{ f \in \Aut(\mathcal{A}) \mid \text{for each } g \in G \text{ there exists } g' \in G \text{ such that } f(\mathcal{A}_g) \subseteq \mathcal{A}_{g'} \right\};
\]
second, the stabilizer of the grading:
\[
\Stab(\Gamma) = \left\{ f \in \Aut(\mathcal{A}) \mid f(\mathcal{A}_g) \subseteq \mathcal{A}_g,\ \forall g \in G \right\};
\]
and finally, the Weyl group of the grading \( \Gamma \),
\begin{equation} \label{eq_Weyl}
\weyl(\Gamma) = \Aut(\Gamma) / \Stab(\Gamma).
\end{equation}
For instance, if \( \mathcal{A} \) is a Lie algebra \( \mathcal{L} \), and \( \Gamma \) is the root space decomposition relative to a Cartan subalgebra of \( \mathcal{L} \), then the Weyl group of the grading coincides with the classical Weyl group generated by reflections with respect to a set of simple roots, possibly extended by automorphisms of the Dynkin diagram. In general, a larger Weyl group reflects a higher degree of symmetry.

  Now assume that \( \Gamma \) is a \( G \)-grading on a complex Lie algebra \( \mathcal{L} \).
A \emph{graded contraction} of \( \Gamma \) is a map \( \varepsilon \colon G \times G \to \mathbb{F} \) such that the vector space \( \mathcal{L} \), endowed with the bracket \( [x, y]^\varepsilon := \varepsilon(g, h)[x, y] \) for \( x \in \mathcal{L}_g \), \( y \in \mathcal{L}_h \), \( g,h \in G \), is again a Lie algebra. In this case, we denote it by \( \mathcal{L}^\varepsilon := (\mathcal{L}, [\cdot, \cdot]^\varepsilon) \).
Note that \( \mathcal{L}^\varepsilon \) inherits a natural \( G \)-grading given by \( (\mathcal{L}^\varepsilon)_g = \mathcal{L}_g \). 
We say that two graded contractions \( \varepsilon \) and \( \varepsilon' \) are \emph{equivalent}, and write \( \varepsilon \sim \varepsilon' \), if the corresponding graded Lie algebras are isomorphic as graded algebras. That is, there exists a Lie algebra isomorphism \( f \colon \mathcal{L}^\varepsilon \to \mathcal{L}^{\varepsilon'} \) such that for every \( g \in G \), there exists \( h \in G \) with \( f(\mathcal{L}_g) = \mathcal{L}_h \).
\smallskip

 It is easy to check that an arbitrary map \( \varepsilon \colon G \times G \to \mathbb{F} \) is a graded contraction of a \( G \)-grading \( \Gamma \) (see, for instance,  \cite[Remark~2.4]{draper2024gradedg2}) if and only if the following conditions hold:
\begin{enumerate}
\item[(a1)] \( \big(\varepsilon(g, h) - \varepsilon(h, g)\big)[x, y] = 0, \)
\item[(a2)] \( \big(\varepsilon(g, h, k) - \varepsilon(k, g, h)\big)[x, [y, z]] 
+ \big(\varepsilon(h, k, g) - \varepsilon(k, g, h)\big)[y, [z, x]] = 0, \)
\end{enumerate}
for all \( g, h, k \in G \) and any choice of homogeneous elements \( x \in \mathcal{L}_g \), \( y \in \mathcal{L}_h \), \( z \in \mathcal{L}_k \). Here, \( \varepsilon \colon G \times G \times G \to \mathbb{F} \) denotes the ternary map defined by
\(
\varepsilon(g, h, k) := \varepsilon(g, h + k)\varepsilon(h, k).
\)

In general, the conditions (a1) and (a2) depend heavily on the specific $G$-grading \( \Gamma \) on the Lie algebra \( \mathcal{L} \). However, there are conditions that guarantee that a given map is always a graded contraction, regardless of the $G$-grading.
Specifically, if the map \( \varepsilon \colon G \times G \to \mathbb{F} \) satisfies the following conditions:
\begin{enumerate}
\item[(c1)] \( \varepsilon(g, h) = \varepsilon(h, g) \),
\item[(c2)] \( \varepsilon(g, h, k) = \varepsilon(k, g, h) \),
\end{enumerate}
for all \( g, h, k \in G \), we say that \( \varepsilon \) is a \emph{generic graded contraction}.  
It is clear that any generic graded contraction defines a valid graded contraction of any \( G \)-grading \( \Gamma \) on any Lie algebra \( \mathcal{L} \).
Remarkably, the converse is also true: according to \cite[Proposition~5]{kochetov2025generic}, any map \( \varepsilon \colon G \times G \to \mathbb{F} \) that defines a graded contraction for every \( G \)-grading \( \Gamma \) on any Lie algebra \( \mathcal{L} \) must be generic (hence the name).

One way to construct generic graded contractions for a group \( G \) is by using 2-co\-boun\-da\-ries (for cohomology with trivial action). 
More precisely, 
if \( \alpha \colon G \to \mathbb{F}^\times \), then the map defined by
\[
\delta\alpha(g,h) = \frac{\alpha(g)\alpha(h)}{\alpha(g+h)}
\]
is a generic graded contraction; that is, \( \varepsilon = \delta\alpha \) satisfies conditions (c1) and (c2).

These 2-coboundaries are also useful in the context of any fixed grading \( \Gamma \), not only because they define graded contractions themselves, but also because of how they act on arbitrary ones. Specifically, given any graded contraction \( \varepsilon \) of \( \Gamma \), the product \( (\delta\alpha)\varepsilon \) is again a graded contraction.
Moreover, the corresponding graded Lie algebra is isomorphic to \( \mathcal{L}^\varepsilon \) as a \( G \)-graded algebra, via the map  
\( f \colon \mathcal{L}^{(\delta\alpha)\varepsilon} \to \mathcal{L}^\varepsilon \) defined by  
\( f(x) = \alpha(g) x \) for every \( g \in G \) and \( x \in \mathcal{L}_g \).
In \cite{draper2024gradedg2} (among others), the contractions \( \varepsilon \) and \( (\delta\alpha)\varepsilon \) are said to be \emph{equivalent by normalization}.

A first tool for the classification of graded contractions is the concept of the \emph{support} of a graded contraction. For any map \( \varepsilon \colon G \times G \to \mathbb{F} \), we define its support as the set
\begin{equation*} 
    S^\varepsilon = \{ (g,h) \in G \times G : \varepsilon(g,h) \ne 0 \}.
\end{equation*}
For example, if \( \varepsilon = \delta\alpha \) is a 2-coboundary, then the support is the entire \( G \times G \). 
Graded contractions with smaller supports yield contracted  Lie algebras that are more abelian.

\begin{remark}
A priori, supports consist of ordered pairs. However, when \( \varepsilon \) is symmetric---that is, condition (c1) holds---we have, in particular, that \( \varepsilon(g,h) \ne 0 \) if and only if \( \varepsilon(h,g) \ne 0 \), so it suffices to work with unordered pairs, which we denote by \( \{g,h\} \).
This will be the situation throughout the present work, and we prefer to consider
\begin{equation}\label{eq_defsupport} 
    S^\varepsilon := \{ \{g,h\}   : g,h\in G,\varepsilon(g,h) \ne 0 \}.
\end{equation}
\end{remark}

In this work, we are interested in obtaining graded contractions of \( \mathbb{Z}_2^3 \)-gradings of the four exceptional Lie algebras other than \( \mathfrak{g}_2 \).  
At first glance, this does not seem like a task that can be approached simultaneously for all cases, since---as already mentioned---the conditions (a1) and (a2) depend heavily on the particular grading of \( \mathcal{L} \).  
However, there is precedent for this approach: in \cite{draper2024gradedb3d4}, results originally developed for \( \mathfrak{g}_2 \) in \cite{draper2024gradedg2} were successfully extended to the cases of \( \mathfrak{b}_3 \) and \( \mathfrak{d}_4 \).
Each graded contraction arising from any of these three \( \mathbb{Z}_2^3 \)-gradings was shown to be equivalent to an \emph{admissible graded contraction} \( \varepsilon \), satisfying \( \varepsilon(g, g) = \varepsilon(g, e) = \varepsilon(e, g) = 0 \). (This condition stems from the fact that \( \mathcal{L}_e = 0 \).) This notion led to the definition of the \emph{nice set} (see Section~\ref{se_nice}), a combinatorial object that precisely characterizes the possible supports of admissible graded contractions.
Despite the specific nature of the arguments used in those cases, we were able to apply them to a variety of \( \mathbb{Z}_2^3 \)-gradings. However, this approach will no longer be valid for the \( \mathbb{Z}_2^3 \)-gradings considered in the present article, for which only a small subset of graded contractions are equivalent to admissible ones.
Instead, we will show that all graded contractions of our four \( \mathbb{Z}_2^3 \)-gradings are generic, and we will associate to each of them a new combinatorial object (introduced in \cite{draper2024generalised}) that precisely describes its support.


 \subsection{Nice sets, \( \mathfrak{g}_2 \), \( \mathfrak{b}_3 \), and \( \mathfrak{d}_4 \)}\label{se_nice}

In this work, we focus on \( \mathbb{Z}_2^3 \)-gradings on exceptional Lie algebras,   so we will always take \( G = \mathbb{Z}_2^3 \). Furthermore, we label the elements of \( G \) as follows:  
\begin{equation}\label{eq_labelsofZ23}
  \begin{array}{cccc}
g_0 := (0, 0, 0), & \quad g_1 := (1, 0, 0), & \quad g_2 := (0, 1, 0), & \quad g_3 := (0, 0, 1), \\
g_4 := (1, 1, 1), & \quad g_5 := (1, 1, 0), & \quad g_6 := (1, 0, 1), & \quad g_7 := (0, 1, 1).
\end{array}  
\end{equation}
The group operation in \( G \) induces a binary operation \( * \) on the index set 
  \(  I_0 := \{0, \dots, 7\} \).  
To follow the notation in \cite{draper2024gradedg2}, we will also use \( I := \{1, \dots, 7\} \), so that \( I_0 = I \cup \{0\} \).   
For \( i, j \in I_0 \), we define \( i * j \) to be the unique element in \( I_0 \) such that \( g_i + g_j = g_{i * j} \).
A bijection \( \sigma \colon I_0 \to I_0 \) is called a \emph{collineation} if  
\( \sigma(i * j) = \sigma(i) * \sigma(j) \) for all \( i, j \in I_0 \).  
In particular, \( \sigma(0) = 0 \).  
The term reflects the identification of these bijections with the collineations of the (unoriented) Fano plane \( P\mathbb{Z}_2^3 \) (see Figure~\ref{fig:Fano}).  
The collineation group, denoted here by \( S_*(I) \), is a well-known simple group of order 168, isomorphic to the group of automorphisms of \( G \).

We say that \( \{i, j, k\} \subseteq I \) is a \emph{generating triplet} if \( \langle g_i, g_j, g_k \rangle = G \), or equivalently, if the three distinct indices satisfy \( k \ne i * j \). The group \( S_*(I) \) acts transitively on the set of ordered generating triplets: for any two such triplets \( \{i, j, k\} \) and \( \{i', j', k'\} \), there exists a unique \( \sigma \in S_*(I) \) such that \( \sigma(i) = i' \), \( \sigma(j) = j' \), and \( \sigma(k) = k' \). In particular, \( S_*(I) \) is in bijective correspondence with the set of ordered generating triplets.
 
The relevance of this group lies in its coincidence with the Weyl group \( \mathcal{W}(\Gamma_{\mathcal{L}}) \) defined in \eqref{eq_Weyl}, for any of the \( \mathbb{Z}_2^3 \)-gradings on \( \mathcal{L} \in \{ \mathfrak{g}_2, \mathfrak{b}_3, \mathfrak{d}_4 \} \) studied in \cite{draper2024gradedg2,draper2024gradedb3d4}; all three share highly symmetric properties.
There is a close relationship between equivalent graded contractions of 
\(    \Gamma_{\mathfrak{g}_2}, \Gamma_{\mathfrak{b}_3}, \Gamma_{\mathfrak{d}_4}   \)
 and collinear supports. More precisely, any support of an admissible graded contraction 
  ---necessarily symmetric---     
   induces a subset of  
\( X := \{ \{i, j\} : i \ne j,\ i, j \in I \} \) by mapping each pair \( (g_i, g_j) \in S^\varepsilon \) to the unordered pair \( \{i, j\} \in X \).
 
 A subset \( T \) of \( X \) is called \emph{nice} \cite[Definition 3.9]{draper2024gradedg2} if, for any generating triplet \( \{i, j, k\} \subseteq I \), the inclusion of \( \{i, j\}, \{k, i * j\} \in T \) implies that
\begin{equation}\label{eq_P}
    P_{\{i,j,k\}} := \{\{i, j\}, \{j,k\},  \{k,i\}, \{i, j * k\}, \{j, k * i\}, \{k, i * j\}\}
\end{equation}
is contained in \( T \). It is easy to prove that the support of any admissible graded contraction of \( \Gamma \) is a nice set \cite[Proposition~3.10]{draper2024gradedg2}.

Conversely, given a nice set \( T \), we define \( \varepsilon^T(g_i, g_j) = 1 \) if \( \{i, j\} \in T \), and \( \varepsilon^T(g_i, g_j) = 0 \) otherwise.  
According to \cite[Proposition~3.11]{draper2024gradedg2}, this defines an admissible graded contraction for each of the three \( \Gamma_{\mathcal{L}} \)'s previously considered, with support equal to \( T \).
There are exactly 24 nice sets up to collineation, according to \cite[Theorem~3.17]{draper2024gradedg2}. The correspondence \( T \mapsto \varepsilon^T \) relates 23 of these nice sets to 23 non-equivalent graded contractions of \( \Gamma_{\mathfrak{g}_2} \), as shown in \cite[Proposition~4.11]{draper2024gradedg2}. For the orthogonal Lie algebras \( \mathfrak{so}(7) \) and \( \mathfrak{so}(8) \), the results are completely analogous.
 However, this will no longer be the case for the \( \mathbb{Z}_2^3 \)-gradings considered throughout this paper: first, the supports will not necessarily be contained in \( X \); and second, not every nice set will correspond to the support of a graded contraction. 
Remarkably, in this setting, the supports of graded contractions coincide exactly with a new combinatorial object introduced in \cite{draper2024generalised}, namely, the \emph{generalised nice set}. 
(This will be established in Section~\ref{se_generica}.)

 \subsection{Generalised nice sets.}\label{se_GNS}
 
 Consider again the sets \( P_{\{i,j,k\}} \) as in Eq.~\eqref{eq_P}, but now without requiring \( \{i, j, k\} \) to be a generating triplet. In fact, the indices \( i, j, k \in I_0 = I \cup \{0\} \) may even be repeated. Note that the resulting sets \( P_{\{i,j,k\}} \) can contain fewer than six elements; for instance, \( P_{\{0,0,0\}} = \{\{0,0\}\} \) contains only one element.

\begin{defi}
Let \( X_0 := \{ \{i, j\} : i, j \in I_0 \} \).
A subset \( T \subseteq X_0 \) is called a \emph{generalised nice set} if, for any \( i, j, k \in I_0 \), the condition \( \{i, j\},\, \{i * j, k\} \in T \) implies that \( P_{\{i,j,k\}} \subseteq T \).
\end{defi}

As mentioned, we will soon see the role that these new combinatorial sets play in our study of certain graded contractions. Clearly, the action of any \( \sigma \in S_*(I) \) on a generalised nice set \( T \) produces another generalised nice set, given by
\[
\tilde\sigma(T) := \{ \{\sigma(i), \sigma(j)\} : \{i, j\} \in T \}.
\]
The classification up to collineations is much richer than in the case of nice sets, yielding exactly 245 non-collinear generalised nice sets \cite[\S~Conclusion]{draper2024generalised}. 

For convenience of notation, we will usually write \( ij \) instead of the unordered pair \( \{i, j\} \) in \( X_0 \).

\begin{example}\label{ex_nice=GNS_losSi}
Some nice sets in \( X \) are also generalised nice sets. Specifically, we have \( S_0 = \emptyset \), and:
\[
\begin{array}{cccc}
S_1 = \{12\}, & S_4 = \{12,13,14\}, & S_7 = \{12,16,67\}, & S_{10} = \{12,16,27,67\}, \\
S_2 = \{12,13\}, & S_5 = \{12,13,17\}, & S_8 = \{25,36,47\}, & S_{11} = \{12,16,17,26,27\}, \\
S_3 = \{12,67\}, & S_6 = \{12,16,26\}, & S_9 = \{12,16,17,26\}, & S_{12} = \{34,36,37,46,47,67\},
\end{array}
\]
and \( S_{13} = P_{\{1,2,3\}} = \{12,23,31,17,26,35\} \).
\end{example}

 \begin{example}\label{ex_outofX_EFyP}
There are some generalised nice sets entirely contained in \( X_0 \setminus X \), namely:
\begin{itemize}
    \item \( E_J := \{jj : j \in J\} \);
    \item \( F_J := \{00, 0j : j \in J\} \);
    \item \( P_{\{0, j, j\}} = \{00, 0j, jj\} \);
\end{itemize}
for any subset \( J \subseteq I \) and any \( j \in I \). (Many of these sets are collinear in the sense of being in the same collineation class.)
\end{example}

 \begin{example}\label{ex_las7excepciones}
There are exactly seven generalised nice sets (up to collineation) \( T \subseteq X_0 \) such that \( T \cap X \) is not itself a generalised nice set. These are:
\begin{itemize}
    \item the full set \( X_0 \);
    \item \( Y_7 := \{00, 01, 02, 05, 11, 12, 15\} \subset Y_{10} := \{00, 01, 02, 05, 11, 12, 15, 22, 25, 55\} \);
    \item \( Y_7 \subset Y_{11} := \{00, 0i, 1i : i = 1, 2, 3, 5, 6\} \subset Y_{15} := \{00, 0i, 1i : i \in I\} \subset Y_{19} := Y_{15} \cup \{23, 35, 26, 56\} \);
    \item \( Y_{26} := X_0 \setminus \{kl : k, l \in \{3, 4, 6, 7\}\} \).
\end{itemize}
The indices indicate the cardinality of each set.
\end{example}

In \cite[Theorem~6.3]{draper2024generalised}, it is proved that any other generalised nice set must be of the form 
\( S_i \cup E_J \), \( S_i \cup F_J \), or \( S_i \cup P_{\{0,j,j\}} \), with certain restrictions on \( J \subseteq I \) and \( j \in I \). A detailed analysis leads to exactly 245 generalised nice sets up to collineations.
We provide a concise and readable summary of this classification in Theorem~\ref{teo_listaGNS}, which compiles the information originally spread across multiple tables in \cite{draper2024generalised}.

\begin{theorem}\label{teo_listaGNS}
There are exactly 245 generalised nice sets up to collineation:
     \begin{enumerate} 
         \item $Y_{7}$, $Y_{10}$, $Y_{11}$, $Y_{15}$, $Y_{19}$, $Y_{26}$,   and $X_0$;
         \item $  P_{\{0,1,1\}}$, $S_1\cup P_{\{0,3,3\}}$, $S_2\cup P_{\{0,4,4\}}$,
         $S_2\cup P_{\{0,7,7\}}$, $S_3\cup P_{\{0,3,3\}}$, $S_6\cup P_{\{0,7,7\}}$,
         $S_7\cup P_{\{0,4,4\}}$, and $S_{10}\cup P_{\{0,4,4\}}$;
      
         \item $E_J$ and $F_J$ for $J=\emptyset,1,12,123,125,1234,1235,12345,123456,I$;
         \item $S_i\cup E_J$ for
         \begin{itemize}  
             \item[-] $i=1$, $J=\emptyset,1,3,12,13,34,36,123,134,346,136,1234,3467,1236,1367,12346, \\{13467},123467 $;
             \item[-] $i=2 $, $J=\emptyset, 1,2,4,7,12,14,17,23,24,27,47,123,124,127,234,247,147,237,1234,\\1237,1247,2347,12347$;
              \item[-] $i=3 $, $J=\emptyset, 1,3,12,13,16,34,123,126,134,137,136,1234,1267,1236,1346,12346,\\12367,123467 $;
               \item[-] $i= 4$, $J=\emptyset, 1,2,12,23,123,234,1234$;
                \item[-] $i=5 $, $J=\emptyset, 1,2,12,23,123,237,1237$;
                 \item[-] $i=6 $, $J=\emptyset,1,7,12,17,126,127,1267 $;
                  \item[-] $i= 7$, $J=\emptyset,1,2,4,12,14,16,17,24,27,124,126,127,146,247,147,1267,1246, {1247},\\12467 $;
                   \item[-] $i= 8$, $J=\emptyset, 2,23,25,234,235,237,2356,2345,23456,234567$;
                    \item[-] $i= 9$, $J=\emptyset,1,2,7,12,17,26,27,126,127,267,1267 $;
                     \item[-] $i=10 $, $J=\emptyset,1,4,12,14,17,124,126,147,1267,1246,12467 $;
                      \item[-] $i=11 $, $J=\emptyset,1,6,12,16,67,126,167,1267 $;
                       \item[-] $i=12 $, $J=\emptyset, 3,34,346,3467$;
                        \item[-] $i= 13$, $J=\emptyset, 1,12,123$.
         \end{itemize}
         \item $S_i\cup F_J$ for
         \begin{itemize} 
     
          \item[-] $i=1$, $J=\emptyset, 3,34,37,125,347,3467,1235,12567,12356,123456,I$;
           \item[-] $i=2$, $J= \emptyset,4,7,47,12356,123456,123567,I$;
            \item[-] $i=3$, $J=\emptyset, 3,34,12567,123567,I$;
             \item[-] $i=4,5,8,9,11,12,13$, $J=\emptyset, I$;
              \item[-] $i=6$, $J= \emptyset,7,123456,I$;
              \item[-] $i=7,10$, $J= \emptyset,4,123567,I$.
         \end{itemize}
            \end{enumerate}
 \end{theorem}


\section{Graded contractions of exceptional Lie algebras.}  \label{sec_principal}

\subsection{Tits construction and \( \mathbb{Z}_2^3 \)-gradings on the exceptional Lie algebras}\label{se_TitsConstruction}

Here we review the simultaneous construction of all the exceptional simple Lie algebras, a beautiful and striking model proposed by Tits \cite{tits1966algebres}. The material presented below is mostly drawn from \cite[Chapter IV]{schafer2017intro_to_nonassoc}, adapted to our context.

Let \( \mathcal{O} \) denote the complex octonion algebra, spanned over \( \mathbb{F} \) by the basis \( \{e_0, e_1, \dots, e_7\} \), where the multiplication is defined by setting \( e_0 = 1 \) as the unit, \( e_i^2 = -1 \), and \( e_i e_j = \pm e_{i * j} \), with the sign determined by the orientation of the arrows in Figure~\ref{fig:Fano}. For instance, \( e_2 e_5 = e_1 = -e_5 e_2 \).

\begin{figure}[h]
  \centering
  \begin{tikzpicture}[scale=0.7, every node/.style={transform shape}] \label{fano}
\draw [postaction={decoration={markings, mark= at position 0.75 with {\arrowreversed{stealth}}}, decorate}] 
(30:1) -- (210:2);
\draw [postaction={decoration={markings, mark= at position 0.75 with {\arrowreversed{stealth}}}, decorate}]
(150:1) -- (330:2);
\draw [postaction={decoration={markings, mark= at position 0.75 with {\arrowreversed{stealth}}}, decorate}]
(270:1) -- (90:2);
\draw [postaction={decoration={markings, mark= at position .24 with {\arrowreversed{stealth}}}, decoration={mark= at position .57 with {\arrowreversed{stealth}}},decoration={mark= at position .9 with {\arrowreversed{stealth}}}, decorate}]
(90:2)  -- (210:2) -- (330:2) -- cycle;
\draw [postaction={decoration={markings, mark= at position .3 with {\arrowreversed{stealth}}}, decoration={mark= at position .63 with {\arrowreversed{stealth}}},decoration={mark= at position .96 with {\arrowreversed{stealth}}}, decorate}] (0:0)  circle (1);
\draw 
(30:1) node[circle, draw, fill=white]{{$e_2$}} 
(210:2) node[circle, draw, fill=white]{{$e_7$}} 
(150:1) node[circle, draw, fill=white]{{$e_4$}}
(330:2) node[circle, draw, fill=white]{{$e_5$}} 
(270:1) node[circle, draw, fill=white]{{$e_6$}} 
(90:2) node[circle, draw, fill=white]{{$e_1$}} 
(0:0) node[circle, draw, fill=white]{{$e_3$}}; 
\end{tikzpicture}
\vspace{-1\baselineskip}
  \caption{Fano plane}
  \label{fig:Fano}
\end{figure}
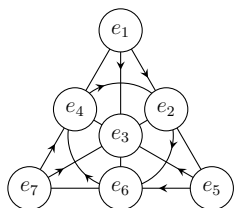

Consider also its quaternion subalgebra $\mathcal{H} := \langle \{e_0, e_1, e_2, e_5\} \rangle$, its two-dimensional subalgebra $\mathcal{K} := \langle \{e_0, e_1\} \rangle$, and its one-dimensional subalgebra $\mathcal{F} := \langle \{e_0\} \rangle$, which can be naturally identified with the base field $\mathbb{F}$. 
We will refer to any ${\mathcal C} \in \{\mathcal{F}, \mathcal{K}, \mathcal{H}, \mathcal{O} \}$ as a \emph{Hurwitz algebra}. All four share some key properties. Let us denote by $I_\mathcal{C} = \{i \in I_0 : e_i \in \mathcal{C} \}$, that is, 
\(
I_{\mathcal{F}} = \{0\} \subseteq 
I_{\mathcal{K}} = \{0, 1\} \subseteq 
I_{\mathcal{H}} = \{0, 1, 2, 5\} \subseteq 
I_{\mathcal{O}} = I_0.
\)
First, ${\mathcal C}$ is endowed with a nonsingular quadratic form (the \emph{norm}) $n\colon {\mathcal C} \to {\mathbb{F}}$, defined by  
\(
n\left(\sum_{i \in I_\mathcal{C}} \alpha_i e_i\right) = \sum_{i \in I_\mathcal{C}} \alpha_i^2,
\)
which is multiplicative; that is, $n(ab) = n(a)n(b)$ for all $a, b \in \mathcal{C}$. In other words, the norm $n$ admits composition, so the algebras are called \emph{composition algebras}.
Moreover, they are \emph{alternative algebras}, meaning that the associator alternates: $(a,a,b) = 0 = (a,b,a)$ for all $a,b \in \mathcal{C}$. Recall that the associator 
\(
(a,b,c) := (ab)c - a(bc)
\)
measures the failure of associativity in the algebra, so that it vanishes  identically for $\mathcal{C} \ne \mathcal{O}$.
 In contrast, the octonions \( \mathcal{O} \) are not associative, as the associator \( (e_i, e_j, e_k) \) is nonzero whenever \( \{i,j,k\} \) forms a generating triplet. 
The map $\text{-} \colon {\mathcal C} \to {\mathcal C}$ given by
\[
a = \alpha_0 e_0 + \sum_{i=1}^{l-1} \alpha_i e_i \mapsto \bar{a} = \alpha_0 e_0 - \sum_{i=1}^{l-1} \alpha_i e_i
\]
is an involution (i.e., a linear map satisfying $\bar{\bar{a}} = a$ and $\overline{ab} = \bar{b}\,\bar{a}$) such that \( a \bar{a} = n(a)1 \).
Note that this involution should not be confused with the usual complex conjugation in the field \( \mathbb{F} \), since here \( \bar{a} = a \) for all \( a \in \mathcal{F} \); that is, the involution restricts to the identity on \( \mathcal{F} \).
The map \( t_{\mathcal C} \colon \mathcal C \to \mathbb{F} \), defined by \( t_{\mathcal C}(a)1 = a + \bar{a} \in \mathcal{F} \), is linear, and every element \( a \in \mathcal{C} \) satisfies the quadratic equation
\(
a^2 - t_{\mathcal C}(a)a + n(a)1 = 0.
\)
In particular, \( t_{\mathcal C}(1) = 2 \) and \( n(1) = 1 \). Let \( \mathcal{C}_0 := \{ a \in \mathcal{C} \colon t_{\mathcal C}(a) = 0 \} \) denote the subspace of traceless elements. Note that \( [a, b] := ab - ba \in \mathcal{C}_0 \) for all \( a, b \in \mathcal{C} \), since \( t_{\mathcal C}(ab) = t_{\mathcal C}(ba) \).
Moreover, $[ \mathcal{C}_0 , \mathcal{C}_0 ]= \mathcal{C}_0$.\smallskip

The derivation algebra \( \mathfrak{der}(\mathcal{C}) := \{ d \colon \mathcal{C} \to \mathcal{C} \colon d(ab) = d(a)b + a d(b) \ \forall\, a,b \in \mathcal{C} \} \) is a Lie subalgebra of \( \mathfrak{gl}(\mathcal{C}) \), as mentioned in Section~\ref{se_gradings}. It is well known that \( \dim \mathfrak{der}(\mathcal{C}) = 0, 0, 3, 14 \), respectively. Moreover, \( \mathfrak{der}(\mathcal{O}) \) is a simple Lie algebra of type \( G_2 \), and \( \mathfrak{der}(\mathcal{H}) \) is a simple Lie algebra of type \( A_1 \) \cite[first column in (4.80)]{schafer2017intro_to_nonassoc}.
Elements of \( \mathfrak{der}(\mathcal{C}) \) can be described using the left and right multiplication operators \( l_a, r_a \colon \mathcal{C} \to \mathcal{C} \), defined by \( l_a(b) := ab \) and \( r_a(b) := ba \) for all \( a, b \in \mathcal{C} \). First, the endomorphism
\[
D_{a,b} := [l_a, l_b] + [l_a, r_b] + [r_a, r_b]
\]
is a derivation of \( \mathcal{C} \). Furthermore,
\[
\mathfrak{der}(\mathcal{C}) = \left\{ \sum_s D_{a_s, b_s} \colon a_s, b_s \in \mathcal{C},\ s \in \mathbb{N} \right\} \equiv D_{\mathcal{C}, \mathcal{C}}.
\]
(Since \( D_{a,b} = -D_{b,a} \) and \( D_{1,a} = 0 \), it is clear that \( \mathfrak{der}(\mathcal{K}) = 0 = \mathfrak{der}(\mathcal{F}) \).)
\smallskip

A commutative algebra \( \mathcal{J} \) satisfying the Jordan identity
\(
(u^2 v) u = u^2 (v u)
\)
is called a \emph{Jordan algebra}. The relevant Jordan algebras for our purposes are the Hermitian matrices
\[
\mathcal{H}_3(\mathcal{C}) := \{ u \in \mathcal{M}_3(\mathcal{C}) \colon \bar{u}^t = u \},
\]
where \( \mathcal{C} \) is any Hurwitz algebra,
 and we denote   \( \overline{(u_{ij})} := (\overline{u_{ij}}) \). 
Its Jordan product   is the so-called \emph{symmetrized} product on \( \mathcal{M}_3(\mathcal{C}) \), defined by
\[
u \cdot v := \frac{1}{2}(uv + vu),
\]
where juxtaposition denotes the usual matrix product.  
 Note that \( u_{ii} \in \mathbb{F} \), and \( \overline{u_{ij}} = u_{ji} \) are arbitrary elements of \( \mathcal{C} \) for \( i < j \), so that
\[
\dim \mathcal{H}_3(\mathcal{C}) = 3 \cdot \dim \mathcal{C} + 3 = 6, 9, 15, \text{ and } 27,
\]
according to \( \dim \mathcal{C} = 1, 2, 4, 8 \), respectively.
Let \( 1 \) denote the \( 3 \times 3 \) identity matrix, and define the trace of \( u \in \mathcal{J}=\mathcal{H}_3(\mathcal{C}) \)
by \( \mathrm{tr}((u_{ij})) := \sum_i u_{ii} \in \mathbb{F} \). Then we have the vector space decomposition
\[
\mathcal{J} = \mathbb{F} \cdot 1 \oplus \mathcal{J}_0, \quad \text{where } \mathcal{J}_0 := \{ u \in \mathcal{J} \colon \mathrm{tr}(u) = 0 \}.
\]
The product
\[
u * v := u \cdot v - \frac{1}{3} \mathrm{tr}(u \cdot v) \cdot 1
\]
defines a commutative multiplication on \( \mathcal{J}_0 \).

Again, the algebra of derivations
\[
\mathfrak{der}(\mathcal{J}) := \{ d \colon \mathcal{J} \to \mathcal{J} \colon d(u \cdot v) = d(u) \cdot v + u \cdot d(v) \ \forall\, u, v \in \mathcal{J} \}
\]
is a Lie algebra, and its elements can be written in terms of multiplication operators. 
In fact, denote by \( R_u \colon \mathcal{J} \to \mathcal{J} \), \( v \mapsto v \cdot u \), the multiplication operator (note that left and right multiplications coincide here). Then
\[
[R_u, R_v] \in \mathfrak{der}(\mathcal{J})
\]
for all \( u, v \in \mathcal{J} \), and these elements span the entire Lie algebra \( \mathfrak{der}(\mathcal{J}) \).
For \( \mathcal{C} \in \{ \mathcal{F}, \mathcal{K}, \mathcal{H}, \mathcal{O} \} \) and \( \mathcal{J} = \mathcal{H}_3(\mathcal{C}) \), the Lie algebra \( \mathfrak{der}(\mathcal{J}) \) is simple of type \( A_1 \), \( A_2 \), \( C_3 \), and \( F_4 \), respectively \cite[first row in (4.80)]{schafer2017intro_to_nonassoc}, with dimensions \( 3 \), \( 8 \), \( 21 \), and \( 52 \).

Consider the Lie algebra
\begin{equation}\label{eq_TitsModel}
\mathcal{T}(\mathcal{C}) := \mathfrak{der}(\mathcal{O}) \oplus (\mathcal{O}_0 \otimes \mathcal{J}_0) \oplus \mathfrak{der}(\mathcal{J}),
\end{equation}
equipped with a Lie bracket \([\, ,\, ]\) defined as follows:
\begin{equation}\label{eq_TitsProduct}
\begin{array}{c}
[\, \mathfrak{der}(\mathcal{O}), \mathfrak{der}(\mathcal{J}) ] = 0, \qquad
[d, a \otimes u] = d(a) \otimes u, \qquad
[D, a \otimes u] = a \otimes D(u), \vspace{4pt}\\
{[}a \otimes u, b \otimes v] = \frac{1}{3} \mathrm{tr}(u \cdot v) D_{a,b} + [a, b] \otimes (u * v) + 2 t_{\mathcal{C}}(ab)[R_u, R_v],
\end{array}
\end{equation}
for all \( d \in \mathfrak{der}(\mathcal{O}) \), \( D \in \mathfrak{der}(\mathcal{J}) \), \( a, b \in \mathcal{O}_0 \), and \( u, v \in \mathcal{J}_0 \).
According to \cite[last row of (4.80)]{schafer2017intro_to_nonassoc}, this construction yields all the exceptional Lie algebras other than \( \mathfrak{g}_2 = \mathfrak{der}(\mathcal{O}) \):
\[
\mathcal{T}(\mathcal{F}) \cong \mathfrak{f}_4, \quad
\mathcal{T}(\mathcal{K}) \cong \mathfrak{e}_6, \quad
\mathcal{T}(\mathcal{H}) \cong \mathfrak{e}_7, \quad
\mathcal{T}(\mathcal{O}) \cong \mathfrak{e}_8.
\]
 \smallskip

We now describe how to equip the four exceptional complex Lie algebras with natural \( G = \mathbb{Z}_2^3 \)-gradings, induced by a fixed grading on the octonion algebra. The octonion algebra \( \mathcal{O} \) is \( \mathbb{Z}_2^3 \)-graded via
\[
\Gamma_{\mathcal{O}} \colon \mathcal{O} = \bigoplus_{g \in G} \mathcal{O}_g, \quad \text{where} \quad \mathcal{O}_{g_i} = \langle \{e_i\} \rangle, \quad i \in I_0.
\]
This grading naturally induces a \( \mathbb{Z}_2^3 \)-grading on \( \mathfrak{der}(\mathcal{O}) \), in accordance with the general framework from Section~\ref{se_gradings}. It is given explicitly by
\begin{equation}\label{eq_gradG2}
\Gamma_{\mathfrak{g}_2} \equiv  \begin{cases}
\mathfrak{der}(\mathcal{O})_{g_0} = 0, \\[2pt]
\mathfrak{der}(\mathcal{O})_{g_i} = \left\{ \sum_{g + h = g_i} D_{a_g, b_h} : a_g \in \mathcal{O}_g,\ b_h \in \mathcal{O}_h \right\}, \quad i \in I.
\end{cases}
\end{equation}
Each nonzero homogeneous component is a 2-dimensional abelian subspace (so a Cartan subalgebra), and the grading enjoys strong symmetry properties: its Weyl group is isomorphic to the full automorphism group \( \Aut(\mathbb{Z}_2^3) \).
Moreover, the  grading is compatible with $\Gamma_{\mathcal{O}}$  in the sense that if \( d \in \mathfrak{der}(\mathcal{O})_{g_i} \), then \( d(e_j) \in \mathcal{O}_{g_{i * j}} \) (besides, \( d(e_j) = 0 \) whenever \( i = j \)). 
As a consequence, not only is the Lie algebra   \( \mathfrak{g}_2=\mathfrak{der}(\mathcal{O}) \) \( \mathbb{Z}_2^3 \)-graded, but all exceptional Lie algebras constructed via the Tits model inherit nice \( \mathbb{Z}_2^3 \)-gradings from the octonions!
Namely, if \( \mathcal{L} = \mathcal{T}(\mathcal{C}) \), then the grading \( \Gamma_{\mathcal{L}} \colon \mathcal{L} = \bigoplus_{g \in G} \mathcal{L}_g \) is defined as follows:
\begin{equation}\label{eq_gradtodas}
 \Gamma_{\mathcal{L}}\equiv \begin{cases}
\mathcal{L}_{g_0} = \mathfrak{der}(\mathcal{J}), \\[2pt]
\mathcal{L}_{g_i} = \mathfrak{der}(\mathcal{O})_{g_i} \oplus (\mathcal{O}_{g_i} \otimes \mathcal{J}_0), \quad i \in I.
\end{cases}
\end{equation}
 Note that the neutral component \( \mathcal{L}_e \) has dimension \( \dim(\mathcal{L}_e) \in \{3, 8, 21, 52\} \), depending on the Hurwitz algebra \( \mathcal{C} \), while each non-neutral homogeneous component has dimension \( \dim( \mathcal{L}_g ) = 2 + \dim \mathcal{J}_0 = 3 \dim \mathcal{C} + 4 \in \{7, 10, 16, 28\} \).
All these gradings \( \Gamma_{\mathcal{L}} \) share certain structural features that are crucial for studying their graded contractions. For example, unlike the case of the \( G \)-grading on \( \mathfrak{g}_2 \), where \( \mathcal{L}_e = 0 \), here \( \mathcal{L}_e \) is a simple Lie algebra.
Regardless of the chosen Hurwitz algebra \( \mathcal{C} \), the associated grading on \( \mathcal{L} = \mathcal{T}(\mathcal{C}) \) exhibits a high degree of symmetry: the Weyl group is as large as possible, 
\[
\mathcal{W}(\Gamma_{\mathcal{L}}) \cong \Aut(G) \cong S_*(I).
\]
Indeed, for any \( \sigma \in S_*(I) \), the map \( f_\sigma \colon \mathcal{L} \to \mathcal{L} \) defined by linear extension of
\begin{equation}\label{eq_extendiendo_colin_a_iso}
D_{e_i,e_j} \mapsto D_{e_{\sigma(i)}, e_{\sigma(j)}}, \quad
e_i \otimes x \mapsto e_{\sigma(i)} \otimes x, \quad
D \mapsto D,    
\end{equation}
for all $i,j\in I$, \( x \in \mathcal{J}_0 \) and \( D \in \mathfrak{der}(\mathcal{J}) \), is an automorphism of the Lie algebra \( \mathcal{L} \).

From now on, we will always work with the group \( G = \mathbb{Z}_2^3 \) and the gradings described above. %
For convenience, we introduce the following notation: 
we write \( \mathcal{D}_i := (\mathfrak{der}(\mathcal{O}))_{g_i} \) and \( \mathcal{M}_i := \mathcal{O}_{g_i} \otimes \mathcal{J}_0 = e_i \otimes \mathcal{J}_0 \) for each \( i \in I \), and
\( \mathcal{L}_i := \mathcal{L}_{g_i} \) for each \( i \in I_0 \).
At some point, we will also use the notation \( \mathcal{M} := \bigoplus_{i \in I} \mathcal{M}_i =\mathcal{O}_0\otimes \mathcal{J}_0\) and \( \mathcal{D} := \bigoplus_{i \in I} \mathcal{D}_i = \mathfrak{der}(\mathcal{O}) \), so that $\mathcal{L}=\mathcal{D}\oplus \mathcal{M}\oplus \mathcal{L}_0   $.

 \subsection{Supports and generalised nice sets}\label{se_generica}
 
Our goal in this section is to prove that the support of any graded contraction of \( \Gamma_{\mathcal{L}} \) is a generalised nice set, and conversely, that any generalised nice set arises as the support of some graded contraction of \( \Gamma_{\mathcal{L}} \).
We begin by showing that every graded contraction of \( \Gamma_{\mathcal{L}} \) must necessarily be generic. This requires a few preliminary lemmas, ensuring, for instance, the existence of sufficiently many linearly independent elements.

 \begin{lemma}\label{le_simetrico}
For any pair of indices \( i, j \in I_0 \), we have \( [\mathcal{L}_i, \mathcal{L}_j] \ne 0 \).
\end{lemma}

 \begin{proof}
 
 This is immediate. If \( i = 0 \), then \( \mathcal{L}_0 = \mathfrak{der}(\mathcal{J}) \) is a simple Lie algebra, so  
\( [\mathcal{L}_0, \mathcal{L}_0] = \mathcal{L}_0 \).  
Moreover, for \( j \ne 0 \), we have \( [\mathcal{L}_0, \mathcal{L}_j] = \mathcal{O}_{g_j} \otimes \mathcal{J}_0 \ne 0 \) by \eqref{eq_TitsProduct}.  
Finally, for distinct nonzero indices \( i, j \), since \( [e_i,e_j] = \pm2e_{i*j} \) and \( \mathcal{J}_0 * \mathcal{J}_0 = \mathcal{J}_0 \), it follows that \( [\mathcal{L}_i, \mathcal{L}_j] \ne 0 \).  \end{proof}

 \begin{lemma}\label{le_eltos}
Let \( \{i, j, k\} \subseteq I \) be a generating triplet, and let \( g, h, f \in G = \mathbb{Z}_2^3 \) be such that
\[
(g, h, f) \in \{ (g_i, e, e),\, (g_i, e, g_k),\, (g_i, e, g_i),\, (g_i, g_i, g_k),\, (g_i, g_j, g_{i*j}),\, (g_i, g_j, g_k) \}.
\]
Then, for any \( \mathcal{C} \in \{\mathcal{F}, \mathcal{K}, \mathcal{H}, \mathcal{O}\} \),
 there exist   \( x \in \mathcal{L}_g \), \( y \in \mathcal{L}_h \), and \( z \in \mathcal{L}_f \) such that the elements \( [x, [y, z]] \) and \( [y, [z, x]] \) are linearly independent in \( \mathcal{L} = \mathcal{T}(\mathcal{C}) \). \end{lemma}

 \begin{proof}
 It is enough to prove the case \( \mathcal{L} = \mathcal{T}(\mathcal{F}) \), since \( \mathcal{T}(\mathcal{F}) \subseteq \mathcal{T}(\mathcal{C}) \) for all possible Hurwitz algebras \( \mathcal{C} \). Note that for \( \mathcal{J} = \mathcal{H}_3(\mathbb{F}) \), the Jordan algebra of symmetric \( 3 \times 3 \) matrices, its derivation algebra becomes \( \f{der}(\mathcal{J}) = [R_{\mathcal{J}}, R_{\mathcal{J}}] = \{\mathrm{ad}(x) : x \in \mathcal{M}_3(\mathbb{F}),\ x^t = -x \} \), which is naturally isomorphic to the Lie algebra of skew-symmetric matrices \( \f{so}(3, \mathbb{F}) \). Here, the map \( \mathrm{ad}(x) \colon \mathcal{J} \to \mathcal{J} \) is defined by \( \mathrm{ad}(x)(u) = xu - ux \), for any \( u \in \mathcal{J} \). This observation will help us identify suitable elements. 
Let \( e_{ij} \in \mathcal{M}_3(\mathbb{F}) \) denote the matrix with a 1 in the \( (i, j) \) position and zeros elsewhere. Define \( e_{ij}^- := e_{ij} - e_{ji} \in \f{so}(3, \mathbb{F}) \), and \( e_{ij}^+ := e_{ij} + e_{ji} \in \mathcal{J} \). 
Note also the relationship between derivations and the associator in the Jordan algebra: since \( \mathcal{J} \) is commutative, we have
\(
[R_u, R_v](w) = u \cdot (v \cdot w) - v \cdot (u \cdot w) = (v, w, u)^{\cdot}.
\)
We now proceed to exhibit convenient homogeneous elements adapted to the various cases of \( (g, h, f) \).

\begin{itemize}
\item[$-$]
Take \( (x,y,z) = (e_i \otimes u, D, D') \in \mathcal{L}_i \times \mathcal{L}_0 \times \mathcal{L}_0 \) as follows:
\[
u = e_{23}^+, \quad
D = \mathrm{ad}(e_{12}^-), \quad
D' = \mathrm{ad}(e_{13}^-).
\]
We compute:
\[
[x, [y, z]] = -e_i \otimes [D, D'](u), \qquad [y, [z, x]] = e_i \otimes D D'(u).
\]
These two elements in \( \mathcal{L}_i \) are linearly independent, because the   elements in \( \mathcal{J}_0 \)   involved are linearly independent:
\[
[D, D'](u) = -2(e_{22} - e_{33}), \qquad D D'(u) = 2(e_{11} - e_{22}).
\]

\item[$-$]
Consider \( (x,y,z) = (e_i \otimes u, D, e_k \otimes v) \in \mathcal{L}_i \times \mathcal{L}_0 \times \mathcal{L}_k \), given by:
\[
u = e_{13}^+, \quad
D = \mathrm{ad}(e_{12}^-), \quad
v = e_{23}^+.
\]
Using Eq.~\eqref{eq_TitsProduct}, we get:
\(
[x, [y, z]] = \tfrac{1}{3} \mathrm{tr}(u \cdot D(v)) D_{e_i, e_k} + [e_i, e_k] \otimes (u * D(v)), 
\)
and
\(
[y, [z, x]] = [e_i, e_k] \otimes D(u * v).
\)
These are linearly independent since their projections to \( \mathcal{O}_0 \otimes \mathcal{J}_0 \) are
linearly independent. To see this, compute
\[
D(u * v) = e_{11} - e_{22}, \qquad u * D(v) = \tfrac{1}{3}(e_{11} - 2e_{22} + e_{33}).
\]

\item[$-$]
Let \( (x,y,z) = (e_i \otimes u, D, e_i \otimes v) \in \mathcal{L}_i \times \mathcal{L}_0 \times \mathcal{L}_i \), with:
\[
u = e_{13}^+, \quad
D = \mathrm{ad}(e_{13}^- + e_{23}^-), \quad
v = e_{23}^+.
\]
We again use Eq.~\eqref{eq_TitsProduct} to get
 $
 [x, [y, z]]=-4[R_{u},R_{D(v)}],
 $
 and
 $
 [y, [z,x]]= -4[D,[R_{v},R_{u}]]=4([R_{D(u)},R_{v}]+[R_{u},R_{D(v)}]).
 $
 These are independent derivations of $\mcal J$, since 
 for an arbitrary element $w=\tiny \left(
\begin{array}{ccc}
 a & b & c \\
 b & d & e \\
 c & e & -a-d \\
\end{array}
\right)\in \mcal J_0$, we compute 
 $$
4(v,w,D(u))
=\tiny   \left(
\begin{array}{ccc}
 2 c & 2 c+e & -2 a-2 b-d \\
 2 c+e & 4 e & -2 a-b-4 d \\
 -2 a-2 b-d & -2 a-b-4 d & -2 (c+2 e) \\
\end{array}
\right)=\mathrm{ad}(e_{13}^-+2e_{23}^-)(w)
$$
and  
$$
4( D(v),w,u)
=\tiny   \left(
\begin{array}{ccc}
 -4 c & -c-2 e & 4 a+b+2 d \\
 -c-2 e & -2 e & a+2 (b+d) \\
 4 a+b+2 d & a+2 (b+d) & 2 (2 c+e) \\
\end{array}
\right)=-\mathrm{ad}(2e_{13}^-+e_{23}^-)(w).
 $$

\item[$-$]
Choose \( (x,y,z) = (e_i \otimes u, e_i \otimes v, e_k \otimes w) \in \mathcal{L}_i \times \mathcal{L}_i \times \mathcal{L}_k \), where:
\[
u = e_{11} - e_{33}, \quad v = w = e_{12}^+.
\]
Using Eq.~\eqref{eq_TitsProduct}, and the facts \( [e_i, e_i e_k] = -2e_k \), and \( D_{e_i, e_k}(e_i) = 4e_k \), we compute:

  $$
  \begin{array}{ll}
  [x,[y,z]]&=-\frac13D_{e_i,e_k}(e_i)\otimes \mathrm{tr}(v\cdot w)u+2\big( \frac13\mathrm{tr}(u\cdot (v*w))D_{e_i,e_ie_k}+[e_i,e_i e_k]\otimes u*(v*w)  \big)\\
  &=-4e_k\otimes \big(\frac13\mathrm{tr}(v\cdot w)u +u*(v*w)\big)+\frac23\mathrm{tr}(u\cdot (v*w))D_{e_i,e_ie_k},\\
 {[}y, [x, z]] &= \text{same with } u \leftrightarrow v.
  \end{array}
  $$
  In order to check that these elements are linearly independent, it is enough to verify that the corresponding tensors in $\mathcal{O}_0 \otimes \mathcal{J}_0$ involve linearly independent elements of $\mathcal{J}_0$. This is clear for our choices of $u$, $v$, and $w$, since
\[\begin{array}{c}
\tfrac{1}{3} \mathrm{tr}(v \cdot w) u + u * (v * w) = \tfrac{1}{3}(2e_{11} - e_{22} - e_{33}), \\
\tfrac{1}{3} \mathrm{tr}(u \cdot w) v + v * (u * w) = \tfrac{1}{6}(e_{11} + e_{22} - 2e_{33})\end{array}
\]
are linearly independent in $\mathcal{J}_0$.

\item[$-$]
Take \( (x,y,z) = (e_i \otimes u, e_j \otimes v, e_{i*j} \otimes w) \in \mathcal{L}_i \times \mathcal{L}_j \times \mathcal{L}_{i*j} \), with:
\[
u = e_{11} - e_{33}, \quad v = e_{12}^+, \quad w = e_{13}^+ + e_{23}^+.
\]
Using \( [e_j, e_{i*j}] = 2e_i \), we obtain:
\(
[y, z] = \tfrac{1}{3} \mathrm{tr}(v \cdot w) D_{e_j, e_{i*j}} + 2e_i \otimes (v * w),
\)
and since \( D_{e_j, e_{i*j}}(e_i) = 0 \), it follows:
    $$
  [x,[y,z]]=2[e_i\otimes u,e_i\otimes(v*w)]=-8[R_u,R_{v*w}]=-4[R_u,R_{w}].
  $$
  Analogously, $[y,[z,x]]=-8[R_v,R_{u*w}].$ These derivations of $\mcal J$ are linearly independent for our choices of $u$, $v$ and $w$, since
  $$
 4[R_u,R_{w}]\colon \tiny \left(
\begin{array}{ccc}
 a & b & c \\
 b & d & e \\
 c & e & -a-d \\
\end{array}
\right)\mapsto   \left(
\begin{array}{ccc}
 4 c & c+2 e & -4 a-b-2 d \\
 c+2 e & 2 e & -a-2 (b+d) \\
 -4 a-b-2 d & -a-2 (b+d) & -2 (2 c+e) \\
\end{array}
\right)
  $$
  while 
  $$
 8[R_v,R_{u*w}]\colon \tiny \left(
\begin{array}{ccc}
 a & b & c \\
 b & d & e \\
 c & e & -a-d \\
\end{array}
\right)\mapsto  \left(
\begin{array}{ccc}
 -2 c & -e & 2 a+d \\
 -e & 0 & b \\
 2 a+d & b & 2 c \\
\end{array}
\right).
  $$
  
  \item[$-$] Finally, the case corresponding to a generating triplet \( \{i,j,k\} \) was already proved in \cite[Lemma 2.2]{draper2024gradedg2}, since \( \mf{der}(\mathcal O) \subseteq \mathcal{T}(\mathcal F) \), and this inclusion is compatible with the \( G \)-gradings considered on both Lie algebras.

\end{itemize}
\end{proof}
With this lemma in mind, it is not difficult to prove that generic graded contractions are the only ones we need to consider in this work.

\begin{prop}
Let \( \varepsilon \colon G \times G \to \mathbb{F} \) be a graded contraction of  
\( \Gamma_{\mathcal{L}} \), the grading defined in \eqref{eq_gradtodas}. Then  
\( \varepsilon \) is a generic graded contraction.
\end{prop}

\begin{proof}
Since the group \( G = \mathbb{Z}_2^3 \) is fixed, we simplify the notation by writing \( \varepsilon_{ij} := \varepsilon(g_i, g_j) \) and \( \varepsilon_{ijk} := \varepsilon_{i, j * k} \varepsilon_{jk} \) for any \( i, j, k \in I_0 \).

By Lemma~\ref{le_simetrico} and the fact that \( \varepsilon \) satisfies condition~\textnormal{(a1)}, we immediately have \( \varepsilon_{ij} = \varepsilon_{ji} \). Since the operation \( * \) is commutative, this implies \( \varepsilon_{ijk} = \varepsilon_{ikj} \).

Recall that condition~\textnormal{(a2)} reads:
\begin{equation*}
\big( \varepsilon_{ijk} - \varepsilon_{kij} \big)[x, [y, z]] + \big( \varepsilon_{jki} - \varepsilon_{kij} \big)[y, [z, x]] = 0,
\end{equation*}
for any \( x \in \mathcal{L}_i \), \( y \in \mathcal{L}_j \), and \( z \in \mathcal{L}_k \). Hence, 
whenever we can find homogeneous elements \( (x, y, z) \in \mathcal{L}_i \times \mathcal{L}_j \times \mathcal{L}_k \) such that
\begin{equation}\label{eq_li}
\{[x, [y, z]], [y, [z, x]]\} \text{ are linearly independent},
\end{equation}
we conclude that \( \varepsilon_{ijk} = \varepsilon_{kij}= \varepsilon_{jki}   \).

Our goal is to show that \( \varepsilon_{ijk} = \varepsilon_{kij} \) for any \( i, j, k \in I_0 \),
 which gives the symmetry required in condition~(c2).
This follows from Lemma~\ref{le_eltos}: there we took
  suitable elements in the spaces
\[
\mathcal{L}_i \times \mathcal{L}_0 \times \mathcal{L}_0 , \
\mathcal{L}_i \times \mathcal{L}_i \times \mathcal{L}_0 ,\ 
\mathcal{L}_i \times \mathcal{L}_0 \times \mathcal{L}_k,\ 
\mathcal{L}_i \times \mathcal{L}_i \times \mathcal{L}_k,\ 
\mathcal{L}_i \times \mathcal{L}_j \times \mathcal{L}_{i * j},\ 
\mathcal{L}_i \times \mathcal{L}_j \times \mathcal{L}_k,
\]
for $\{i,j,k\}$ a generating triplet,
 all of them satisfying \eqref{eq_li}. 
 (Elements satisfying \eqref{eq_li} also exist if $i=j=k=0$, although in this case we do not need them   to conclude  \( \varepsilon_{000} = \varepsilon_{000} \).)
\end{proof}

In some sense, the previous result shows that our gradings \( \Gamma_{\mathcal{L}} \)'s have the maximal number of linearly independent brackets between homogeneous components, and thus only generic graded contractions can occur.  
At first glance, this might suggest that there are fewer graded contractions of \( \Gamma_{\mathcal{L}} \) compared to other cases, since the condition on \( \varepsilon \) is more restrictive.
However, we will eventually see that this is not the case: for instance, the three gradings on \( \mathfrak{g}_2 \), \( \mathfrak{b}_3 \), and \( \mathfrak{d}_4 \) mentioned earlier admit significantly fewer graded contractions.  
To be more precise, although there are more possible maps \( \varepsilon \) that define graded contractions of \( \Gamma_{\mathfrak{g}_2} \), \( \Gamma_{\mathfrak{b}_3} \), and \( \Gamma_{\mathfrak{d}_4} \), the number of equivalence classes  is much smaller than in the present case.

 Consequently, the supports of our graded contractions turn out to be generalised nice sets, in view of the following proposition.

\begin{prop}\label{pr_elsoporteesGNS}
Let \( \varepsilon \colon G \times G \to \mathbb{F} \) be a generic graded contraction.  
Then its support  
\(
 S^\varepsilon := \{ \{g,h\}   : g,h\in G, \varepsilon(g,h) \ne 0 \}
\)
is a generalised nice set.
\end{prop}

\begin{proof}
The support \( S^\varepsilon \) is well defined since \( \varepsilon \) is symmetric.  
Let \( i, j, k \in I_0 \) be such that \( \{i, j\}, \{i * j, k\} \in S \).  
This implies that \( \varepsilon_{i * j,\, k} \varepsilon_{ij} = \varepsilon_{kij} \ne 0 \).  
Since \( \varepsilon_{kij} = \varepsilon_{jki} = \varepsilon_{ijk} \), all three terms are nonzero, and we conclude that  
\( \{i, k\}, \{i * k, j\}, \{j, k\}, \{j * k, i\} \in S^\varepsilon \).  
In other words, \( P_{\{i, j, k\}} \subseteq S^\varepsilon \).
\end{proof}

To complete the picture, we now prove that every generalised nice set is the support of a generic graded contraction.

\begin{prop}\label{prop_ep^T}

If \( S \) is a generalised nice set, then the map \( \varepsilon^S \colon G \times G \to \mathbb{F} \) defined by
\[
\varepsilon^S(g_i,g_j)\equiv 
\varepsilon^S_{ij} =
\begin{cases}
1 & \text{if } \{i,j\} \in S,\\
0 & \text{if } \{i,j\} \notin S,
\end{cases}
\]
is a   generic   graded contraction with support \( S \).
\color{black}
\end{prop}

\begin{proof}
Since \( X_0 \) consists of unordered pairs, we trivially have \( \varepsilon^S_{ij} = \varepsilon^S_{ji} \). Thus, it remains to show that \( \varepsilon^S_{kij} = \varepsilon^S_{ijk} \) for all \( i, j, k \in I_0 \); that is, we must verify that
\[
\varepsilon^S_{k,\, i * j} \, \varepsilon^S_{ij} = \varepsilon^S_{i,\, j * k} \, \varepsilon^S_{jk}.
\]

Observe that both \( \{k, i * j\} \) and \( \{i, j\} \) belong to \( S \) if and only if \( P_{\{i, j, k\}} \subseteq S \), which in turn holds if and only if both \( \{i, j * k\} \) and \( \{j, k\} \) are in \( S \).  
This implies that  either one factor on each side of the above equation is zero---making the equality trivially true---or all four terms are nonzero. In the latter case, all four values equal \( 1 \), so the equality holds as well.
 \end{proof}

\color{black}

In this way, for each Hurwitz algebra \( \mathcal{C} \) and the corresponding simple Lie algebra \( \mathcal{L} = \mathcal{T}(\mathcal{C}) \), and for each generalised nice set \( S \), we obtain a new \( G \)-graded Lie algebra via the graded contraction, \( \mathcal{L}^{\varepsilon^S} \). In the next section, we will analyse the properties of the various algebras that arise in this family.


\section{Properties of the obtained Lie algebras }\label{se_propiedades}

In this section, we study the properties of the Lie algebras obtained in  Proposition~\ref{prop_ep^T}.
Recall that the \emph{lower central series} of a Lie algebra \( \mathcal{L} \) (respectively, the \emph{derived series}) is the sequence of subalgebras defined by \( \mathcal{L}^0 = \mathcal{L} \), \( \mathcal{L}^1 = {\mcal L}'=[\mathcal{L}, \mathcal{L}] \), and \( \mathcal{L}^n = [\mathcal{L}, \mathcal{L}^{n-1}] \) (respectively, \( \mathcal{L}^{(n)} = [\mathcal{L}^{(n-1)}, \mathcal{L}^{(n-1)}] \)) for all \( n \geq 1 \).
The algebra \( \mathcal{L} \) is called \emph{nilpotent} (respectively, \emph{solvable}) if there exists some \( n \) such that \( \mathcal{L}^n = 0 \) (respectively, \( \mathcal{L}^{(n)} = 0 \)). The smallest such integer is called the \emph{nilindex} (respectively, the \emph{solvability index}).
Note that nilpotency implies solvability, but not conversely.
We describe the lower central series (respectively, the derived series) of \( \mathcal{L} \) by a sequence of dimensions
\(
(d_0, d_1, \dots, d_m),
\)
where \( d_i = \dim \mathcal{L}^i \) (respectively, \( d_i = \dim \mathcal{L}^{(i)} \)), and \( d_m = d_\ell \) for all \( \ell \ge m \). That is, the sequence stops when the dimensions stabilize. In particular, the algebra is nilpotent (respectively, solvable) if \( d_m = 0 \).
The \emph{radical} of \( \mathcal{L} \), denoted \( \mathfrak{r}(\mathcal{L}) \), is its largest solvable ideal. The algebra is said to be \emph{semisimple} if \( \mathfrak{r}(\mathcal{L}) = 0 \), which is equivalent to \( \mathcal{L} \) being a direct sum of simple ideals.
Any Lie algebra \( \mathcal{L} \) can be written as a semidirect sum of its radical with a semisimple subalgebra, called a \emph{Levi subalgebra}. This is known as the \emph{Levi decomposition} of \( \mathcal{L} \).
The \emph{center} of \( \mathcal{L} \), denoted by \( \mathfrak{z}(\mathcal{L}) = \{ x \in \mathcal{L} : [x, \mathcal{L}] = 0 \} \), is always contained in the radical. 
The Lie algebra \( \mathcal{L} \) is said to be \emph{reductive} if \( \mathfrak{r}(\mathcal{L}) = \mathfrak{z}(\mathcal{L}) \), that is, if \( \mathcal{L} \) decomposes as the direct sum of a semisimple Lie algebra and an abelian Lie algebra.
 \smallskip
 
 For any generalised nice set \( T \) in Theorem~\ref{teo_listaGNS}, we consider the map \( \varepsilon^T \) as in  Proposition~\ref{prop_ep^T}, and the Lie algebra obtained from the graded algebra \( \mathcal{T}(\mathcal{C}) \) in Eq.~\eqref{eq_gradtodas} by applying the graded contraction   \( \varepsilon^T \). 
We will slightly abuse notation by denoting this contracted Lie algebra by \( \mathcal{L}_T \), referring to \( (\mathcal{T}(\mathcal{C}))^{\varepsilon^T} \), without explicitly mentioning the underlying Hurwitz algebra \( \mathcal{C} \), since many of the properties under consideration do not depend on its choice.
 \smallskip

It is convenient to begin by analysing the bracket of \( \mathcal{L} = \mathcal{T}(\mathcal{C}) \), being slightly more explicit than in Lemma~\ref{le_simetrico}.  
Recall our notation for the different subspaces of \( \mathcal{L} \): for each \( i \in I \), we write  
$\mcal D_i = (\f{der}(\mcal O))_{g_i}$, 
$\mcal D = \bigoplus_{i \in I} \mcal D_i$, 
$\mcal M_i = e_i \otimes \mcal J_0$, 
$\mcal M = \bigoplus_{i \in I} \mcal M_i$,
and \( \mathcal{L}_i = \mathcal{D}_i \oplus \mathcal{M}_i = \mathcal{L}_{g_i} \).

\begin{lemma}\label{le_comovanloscorchetes}
 Let $\mcal L = \mcal T(\mcal C)$. 
 
If $\ell = \dim \mcal C$, then $\dim \mcal D_i = 2$, $\dim \mcal M_i = 3\ell + 2$, and $\dim \mcal L_i = 3\ell + 4$ for all $i \ne 0$.

For all $i \ne j$ in $I$, we have $[\mcal L_0, \mcal D_i] = 0$ and $[\mcal L_0, \mcal M_i] = \mcal M_i$, and:
\begin{equation}\label{eq_varioscorchetes}
\begin{array}{ll}
[\mcal D_i, \mcal D_i] = 0, & [\mcal D_i, \mcal D_j] = \mcal D_{i \ast j}, \\
{[}\mcal D_i, \mcal M_i] = 0, & [\mcal D_i, \mcal M_j] = \mcal M_{i \ast j}, \\
{[}\mcal M_i, \mcal M_i] = \mcal L_0, \quad\qquad& [\mcal M_i, \mcal M_j] = \mathbb{F} D_{e_i, e_j} \oplus \mcal M_{i \ast j} \subsetneq \mcal L_{i \ast j}.
\end{array}
\end{equation}

In addition, if $j \ne k$ in $I$, then:
\begin{itemize}
    \item[\rm(a)] $\{x \in \mcal L_0 : [x, \mcal L_0] = 0\} = 0$;
    \item[\rm(b)] $\{x \in \mcal L_0 : [x, \mcal L_j] = 0\} = 0$;
    \item[\rm(c)] $\{x \in \mcal L_j : [x, \mcal L_j] = 0\} = \mcal D_j$;
    \item[\rm(d)] $\{x \in \mcal L_j : [x, \mcal L_k] = 0\} = 0$.
\end{itemize}
\end{lemma}

 \begin{proof}
 The first statements are straightforward, always keeping in mind the product given in~\eqref{eq_TitsModel}. To prove~\eqref{eq_varioscorchetes}, recall that the identities
\(
[\mathcal{D}_i, \mathcal{D}_i] = 0\) and \( [\mathcal{D}_i, \mathcal{D}_j] = \mathcal{D}_{i \ast j}
\)
are well known, see for instance~\cite{elduque2013gradings_simple_lie} or~\cite{draper2024gradedb3d4}. Since \( d(e_i) = 0 \) for every \( d \in \mathcal{D}_i \), it follows that
\(
[\mathcal{D}_i, \mathcal{M}_i] = 0.
\)
For any \( j \ne i \), there exists \( d \in \mathcal{D}_i \) such that \( d(e_j) \ne 0 \), so that 
\(
[\mathcal{D}_i, \mathcal{M}_j] = \mathcal{M}_{i \ast j}.
\)
 Next, we consider the brackets  $[\mcal M_i,\mcal M_i]$. For any $u,v \in \mcal J_0$ we have
\[
[e_i \otimes u, e_i \otimes v] = 2t_{\mathcal{C}}(e_i^2)[R_u, R_v] = -4[R_u, R_v] \in \f{der}(\mcal J) = \mcal L_0.
\]
 Since the space $\mcal L_0 = \f{der}(\mcal J)$ is spanned by such brackets $[R_u, R_v]$ (with $u, v \in \mcal J_0$), we conclude that $[\mcal M_i, \mcal M_i] = \mcal L_0$.
Now, for $i \ne j$, we can always choose the ordering so that $e_i e_j = e_{i \ast j}$. From the multiplication rule~\eqref{eq_TitsProduct}, we obtain:
\begin{equation}\label{eq_lemita}
[e_i \otimes u, e_j \otimes v] = \tfrac{1}{3} \mathrm{tr}(u \cdot v)\, D_{e_i,e_j} + 2e_{i \ast j} \otimes (u \ast v) \in \mathbb{F} D_{e_i,e_j} \oplus \mcal M_{i \ast j}.
\end{equation}
To show the converse inclusion, take $u,v \in \mcal J_0$ as in the proof of Lemma~\ref{le_eltos}, and compute:
\[
[e_i \otimes (e_{11} - e_{33}),\ e_j \otimes (e_{22} - e_{33})] + [e_i \otimes e_{12}^+,\ e_j \otimes e_{12}^+] = D_{e_i,e_j}.
\]
This gives $D_{e_i,e_j} \in [\mcal M_i, \mcal M_j]$. Together with~\eqref{eq_lemita}, this implies that $e_{i \ast j} \otimes (u \ast v) \in [\mcal M_i, \mcal M_j]$ for all $u,v \in \mcal J_0$. Since $\mcal J_0 \ast \mcal J_0 = \mcal J_0$, it follows that
\(
\mcal M_{i \ast j} = e_{i \ast j} \otimes \mcal J_0 \subseteq [\mcal M_i, \mcal M_j],
\)
which finishes the proof of~\eqref{eq_varioscorchetes}.

The claims  (a) to (d) are relatively clear. Assertion (a), for instance, follows from 
 the fact that $\f{z}(\mcal L_0) = 0$, since $\mcal L_0$ is a simple Lie algebra. 
For (b), take $0 \ne D \in \f{der}(\mcal J)= \mcal L_0$. Then there exists $u \in \mcal J$ such that $D(u) \ne 0$. We may assume $u \in \mcal J_0$, because $D(1) = 0$. Hence, $[D, e_j \otimes u] = e_j \otimes D(u) \ne 0$.
To prove (c), suppose $[e_j \otimes u,\ \mcal L_j] = 0$. In particular, $[e_j \otimes u,\ e_j \otimes v] = -4[R_u, R_v] = 0$ for all $v \in \mcal J_0$. That is, $(u, \mcal J_0, \mcal J_0) = 0$, so $u$ associates with all elements of $\mcal J$. This forces $u = 0$, since the only elements in $\mcal J$ that associate with all others are scalar multiples of the identity, and $u \in \mcal J_0$ is traceless.
Finally, for (d), suppose $d \in \mcal D_j$ and $u \in \mcal J_0$ satisfy $[d + e_j \otimes u,\ \mcal L_k] = 0$. Take $d' \in \mcal D_k\subseteq \mcal L_k$ such that $d'(e_j) \ne 0$. Since
\(
[d,\ d'] + d'(e_j) \otimes u = 0,
\)
then $u = 0$. It follows that $[d,\ \mcal D_k] = 0$, and hence $d = 0$.
 \end{proof}
 
 Now we study some properties of the graded contraction \( \varepsilon^T \) associated with a generalised nice set \( T \). 
 This analysis will help us to distinguish how many non-isomorphic algebras arise (at most 245, a priori, since collineations induce automorphisms via Eq.~\eqref{eq_extendiendo_colin_a_iso}), and to illustrate the diversity of Lie algebras appearing in this family.
 For brevity, we denote by $\mathfrak{z} = \mathfrak{z}(\mathcal{L}_T)$ the center, by $\mathfrak{r} = \mathfrak{r}(\mathcal{L}_T)$ the radical, and by $\mathcal{L}_{ss}$ the Levi subalgebra of the Lie algebra $\mathcal{L}_T = \mathfrak{r} \rtimes \mathcal{L}_{ss}$ (where \emph{ss} stands for \emph{semisimple}). 
Note that if $\mathcal{L}_T$ admits a vector space decomposition $\mathcal{L}_T = R \oplus L$, where $L$ is a semisimple Lie subalgebra and $R$ is a solvable ideal, then necessarily $R = \mathfrak{r}(\mathcal{L}_T)$ and $L = \mathcal{L}_{ss}$.
Let us write $\ell = \dim \mathcal{C}$, so that $3\ell + 2 = \dim \mathcal{M}_i$ and $\tilde{\ell} := 3\ell + 4 = \dim \mathcal{L}_i$. 
The results below are organized according to the examples of generalised nice sets described in Section~\ref{se_GNS}.

We begin by studying the generalised nice sets contained in \( X \) (and therefore nice sets).  
Recall from Example~\ref{ex_nice=GNS_losSi} that these are given by \( T = S_i \), for \( i = 0, \dots, 13 \).
Let us define the following index sets:
\[ 
K_i := \{ j \in I : \{k, j\} \in S_i \text{ for some } k \},  \qquad \color{black}
J_i := \{ k * j : \{k, j\} \in S_i \}.
\]
These subsets of \( I \) are closely related to the center and the derived algebra of \( \mathcal{L}_T \).
We can determine directly the sets \( K_i \) and \( J_i \) associated to each generalised nice set \( S_i \subseteq X \) as follows:

\begin{table}[h]
    \centering
    \begin{tabular}{|c||c|c|c|c|c|c|c|c|c|c|c|c|}
        \hline
        $i$ & 0 &1 & 2 & 3 & 4 & 5 & 6 & 7,\,10 & 8 & 9,\ 11 & 12 & 13 \\
        \hline
        $K_i$ &$\emptyset$& 12 & 123 & 1267 & 1234 & 1237 & 126 & 1267 & 234567 & 1267 & 3467 & 123567 \\
        \hline
        $J_i$ &$\emptyset$& 5 & 56 & 5 & 567 & 456 & 345 & 35 & 1 &345 & 125 & 4567 \\
        \hline
    \end{tabular}
    \vspace{0.18em} 
    \caption{Sets \( K_i \) and \( J_i \) associated to the generalised nice sets \( S_i \subseteq X \)}
    \label{tab:table_KiJi}
\end{table}

\begin{prop}\label{pr_casolosSi}
    If \( T \) is a generalised nice set contained in \( X \), that is, \( T = S_i \) as in Example~\ref{ex_nice=GNS_losSi}, then
    \[
        \mathfrak{z} = \mathfrak{r} = \mcal L_0 \oplus \bigoplus_{k \in I\setminus K_i} \mcal L_k, \qquad
        [\mcal L_T, \mcal L_T] = \bigoplus_{k \in J_i} \mcal L_k.
    \]
    \begin{itemize}
        \item If \( T = S_0 = \emptyset \), then \( \mcal L_T \) is abelian.
        \item If \( T = S_i \) for some \( i \ne 0 \), then \( \mcal L_T \) is:
        \begin{itemize}
            \item 2-step solvable for all \( i \ne 0 \);
            \item 2-step nilpotent for all \( i \ne 0, 13 \);
            \item 3-step nilpotent if \( i = 13 \).
        \end{itemize}
         \end{itemize}
    Among these algebras, all are pairwise non-isomorphic except in one case: \( \mcal L_{S_4} \cong \mcal L_{S_5} \).
\end{prop}

\begin{proof}
The center is a homogeneous subspace, 
\( \mathfrak{z}(\mcal L_T) = \bigoplus_{k \in I_0} \mcal L_k \cap \mathfrak{z}(\mcal L_T) \). 
Now, if \( k \notin K_i \), then \( kj \notin T \) for all \( j \), so \( \varepsilon^T_{kj} = 0 \). Hence, 
\( [\mcal L_k, \mcal L_j]^{\varepsilon^T} = \varepsilon^T_{kj} [\mcal L_k, \mcal L_j] = 0 \), 
meaning \( \mcal L_k \subseteq \mathfrak{z}(\mcal L_T) \). In particular, this applies to \( \mcal L_0 \), since \( 0 \notin K_i \). 
Conversely, suppose \( k \in K_i \). Then there exists \( j \in I \) such that \( kj \in T \). 
By Lemma~\ref{le_comovanloscorchetes}(d), we know that any nonzero \( x \in \mcal L_k \) satisfies \( [x,\mcal L_j] \ne 0 \). 
Since \( \varepsilon^T_{kj} \ne 0 \), we also have \( [x,\mcal L_j]^{\varepsilon^T} \ne 0 \), so \( x \notin \mathfrak{z}(\mcal L_T) \). 
That is, \( \mcal L_k \cap \mathfrak{z}(\mcal L_T) = 0 \), and hence 
\( \mathfrak{z}(\mcal L_T) = \mcal L_0 \oplus \bigoplus_{k \in I \setminus K_i} \mcal L_k \).

Next we focus on solvability and nilpotency. Observe in Table~\ref{tab:table_KiJi} that \( K_i \cap J_i = \emptyset \) for all \( i \ne 13 \). This means that \( J_i \subseteq I \setminus K_i \), so that 
\(
[\mcal L_T, \mcal L_T] \subseteq \f z,
\)
and consequently
\(
(\mcal L_T)^2 = [\mcal L_T, [\mcal L_T, \mcal L_T]] = 0:
\)
all these algebras $\mcal L_{S_i}$, \( i \ne 13 \), are nilpotent of index 2. (In particular, they are solvable as well, and abelian only for \( i = 0 \).)
The case \( i = 13 \) differs from the rest in that \( K_{13} = I \setminus\{4\} \) intersects \( J_{13} = \{4,5,6,7\} \), so that
the derived algebra
$[\mcal L_T, \mcal L_T] $ is not central. We compute
\(
[\mcal L_T, [\mcal L_T, \mcal L_T]] = \mcal L_4 \subseteq \f z.
\)
This implies that the first vanishing term in the lower central series is \( (\mcal L_T)^3 = 0 \). 
It is also immediate that \( (\mathcal{L}_T)^{(2)} = 0 \), as there are no indices \( k, j \in J_{13} \) for which \( kj \in S_{13} \).

  It remains to determine for which indices the corresponding algebras are graded-isomorphic. 
Recall that the Lie algebras  \( \mcal D^{\varepsilon^{S_i}} \) associated to the graded contractions \( \varepsilon^{S_i} \), when applied to the grading on \( \mcal D = \f{der}(\mcal O) \), were proved to be pairwise non-isomorphic in \cite[Theorem~3.27]{draper2024gradedg2}, 
except for the following case: \( S_4 = \{12,13,14\} \) and \( S_5 = \{12,13,17\} \) (corresponding to \( T_8 \) and \( T_{10} \) in that work). 
Since \( \mcal D \) is a graded subalgebra of \( \mcal L \), it follows that the Lie algebras \( \mcal L_{S_i} \) and \( \mcal L_{S_j} \) are also non-isomorphic whenever \( i \ne j \), except possibly when \( \{i,j\} = \{4,5\} \). 
On the other hand, showing that \( \mcal L_{S_4} \cong \mcal L_{S_5} \) requires providing an explicit graded isomorphism between the two, and cannot be deduced directly from \cite{draper2024gradedg2}.

    Take the three derivations $x_i\in\f{der}(\mathcal{O})$ determined by $x_i(e_1)=0=x_i(e_4)$ and $x_i(e_3)=\frac12e_ie_3$ for $i=1,7,4$. (Since $\{1,3,4\}$ is a generating triplet, this uniquely determines the derivation.) Take also 
    $$
    y_1=\frac14D_{e_7,e_4},\quad
      y_7=\frac14D_{e_4,e_1},\quad
     y_4=\frac14D_{e_1,e_7}.
    $$
    Note that $\mcal D_i=\langle x_i,y_i\rangle$ for any $i=1,7,4$. Moreover,
    $[x_i,x_j]=x_{i*j}$, $[x_i,y_j]=0$, $[y_i,y_j]=y_{i*j}$ for the ordered pairs $(i,j)=(1,7),(7,4),(4,1)$. 
    As in \cite[Eq.~(16)]{draper2024gradedg2}, the map $\varphi\colon\mcal D\to\mcal D$ given by $\varphi\vert_{\mcal D_k}=\id$  for $k=2,3,5,6$ and 
    $$
    \begin{array}{ccc}
    \varphi(x_1)=x_1,  &\varphi(x_7)=x_4,&\varphi(x_4)=-x_7,\\
     \varphi(y_1)=y_1, & \varphi(y_7)=y_4,&\varphi(y_4)=-y_7,
     \end{array}
    $$
    is a graded isomorphism $\varphi\colon \mcal D^{\ep_{S_4}}\to\mcal D^{\ep_{S_5}}$.
    Our mission is to extend this isomorphism 
    to a graded isomorphism $\tilde \varphi\colon \mcal L_{ {S_4}}\to\mcal L_{ {S_5}}$. 
    We simply consider $\tilde \varphi\vert_{\mcal D}=\varphi$, $\tilde \varphi\vert_{\f{der}(\mcal J)}=\id$, 
    and $\forall u\in\mcal J_0$,
    $$
    \tilde \varphi(e_7\otimes u)=e_4\otimes u,\quad
    \tilde \varphi(e_4\otimes u)=-e_7\otimes u,\quad
\tilde \varphi(e_k\otimes u)=e_k\otimes u \quad\forall k=1,2,3,5,6.
    $$
    We can verify that this defines the required isomorphism by direct substitution in Eq.~\eqref{eq_TitsProduct}. 
The verification is straightforward, taking into account the following facts: 
the Lie bracket is only non-zero for homogeneous elements whose degrees belong to the supports \( S_4 \) and \( S_5 \); 
for any \( d \in \mcal D_1 \), we have \( d(e_1) = d(e_4 e_7) = 0 \), so there exists \( \alpha \in \mathbb{F} \) such that \( d(e_7) = \alpha e_4 \); 
moreover, \( d(e_4) = -\alpha e_7 \); 
and finally, \( D_{e_i, e_k}(e_i) = 4e_k \).
 \end{proof}
 
 The results obtained so far do not differ substantially from those in the case of \( \mathfrak{g}_2 \) studied in~\cite{draper2024gradedg2}, except in quantity. However, this similarity does not carry over to the other generalised nice sets. In particular, many of the algebras arising in the next families will no longer be solvable; that is, they will have a nontrivial semisimple part \( \mathcal{L}_{ss} \ne 0 \).

Let us proceed by following  the examples in Section~\ref{se_GNS}. 
Example~\ref{ex_outofX_EFyP} introduces three types of generalised nice sets contained in 
\( X_0 \setminus X \), namely:
\[
E_J := \{jj : j \in J\}, \qquad 
F_J := \{00, 0j : j \in J\}, \qquad 
P_{\{0,j,j\}} := \{00, 0j, jj\},
\]
for subsets \( J \subseteq I \) and \( j \in I \). We will show that the structure of the Levi factor allows us to distinguish among  these three cases.
(To be precise, it will be \( \mathcal{L}_{ss}= 0, \f{der}(\mcal J) \) and $\f{innstr}(\mcal J)$, respectively.)

 \begin{prop}\label{pr_losEsyFs}
    Let $\emptyset \ne J \subseteq I$. 
     \begin{itemize}
         \item If $T=E_J$, then $\mcal L_T=\f r$ is 2-step nilpotent and not reductive, since $\f z=\f{der}(\mcal J)\oplus\f{der}(\mcal O)\oplus \sum_{j\in I\setminus J}\mcal M_j.$
          \item If $T=F_{\emptyset}=\{00\}$, then $\mcal L_T$ is   reductive, with Levi factor $\mcal L_{ss}=\mcal L_0\cong\f{der}(\mcal J)$ and radical $\f r=\sum_{i\in I}\mcal L_i=\mcal D\oplus \mcal M=\f z$, an abelian algebra. 
          \item If $T=F_J$, then $\mcal L_T$ is not  reductive, with Levi factor $\mcal L_{ss}=\mcal L_0\cong\f{der}(\mcal J)$ and radical $\f r=\sum_{i\in I}\mcal L_i=\mcal D\oplus \mcal M$, an abelian algebra. Here $\f z=\f{der}(\mcal O)\oplus\sum_{j\in I\setminus J}\mcal M_j$.
     \end{itemize}
     Moreover, 
     $\mcal L_{E_J}\cong \mcal L_{E_K}$ if and only if $|J|=|K|$, $\mcal L_{F_J}\cong \mcal L_{F_K}$ if and only if $|J|=|K|$, and $\mcal L_{E_J}\cong \mcal L_{F_K}$  cannot occur. 
 \end{prop}

 Clearly, the algebras \( \mathcal{L}_{E_J} \)'s and \( \mathcal{L}_{F_K} \)'s are very different: the former are solvable, whereas the latter are not ($\mcal L_{ss}\ne0$).
 
To avoid confusion, whenever we use the indices \( i, j, k \) in the subsequent pages, they will be understood to range over \( I \) (not over \( I_0 \)).

 \begin{proof}
 Assume first that $T=E_J$. Clearly $\mcal L_0 $ and $  \mcal L_j $ are contained in $ \f z $ if $j\notin J$,  
    since  $\ep^T_{jk}=0$ for all $k$.
    Also, if $j\in J$, $\mcal D_j\subseteq \f z$ by Lemma~\ref{le_comovanloscorchetes}, since $[\mcal D_j,\mcal L_j]=0$  and 
     $\ep^T_{jk}=0$ for all $k\ne j$. Besides, $\f z$ is a homogeneous subspace and, if $j\in J$, $\mcal L_j\cap \f z\subseteq \mcal D_j$ by item (c) in that lemma. Thus, the center equals $\f z=\mcal L_0\oplus \sum_{j\in I\setminus J}\mcal L_j\oplus \sum_{j\in J}\mcal D_j= \f{der}(\mcal J)\oplus\f{der}(\mcal C)\oplus \sum_{j\in I\setminus J}\mcal M_j$. 
     If we compute the derived algebra, $[\mcal L_T,\mcal L_T]=\mcal L_0$. Since $00,0j\notin T$, then $\mcal L_T^2=\mcal L_T^{(2)}=0$ and the algebra is 2-step nilpotent and 2-step solvable (not abelian).
     
Second, consider $T=F_J$. We begin with the case $J=\emptyset$, i.e. $T=\{00\}$,    since it has different properties from the others. The assertions follow straightforwardly. As before $[\mcal L_T,\mcal L_T]=\mcal L_0$, but now $00\in T$, so that $\mcal L_T^n=\mcal L_T^{(n)}=\mcal L_0$ for all $n$. 
 The center is $\f z=\oplus_{i\in I} \mcal L_i=\mcal D\oplus \mcal M$  and $\mcal L_T=\mcal L_0\oplus\f z$. Since $\mcal L_0$ is simple, then the center coincides with the 
 radical (which in particular is abelian). If $J\ne\emptyset$, again the subspace $ R:=\oplus_{i\in I} \mcal L_i$  is abelian ($ii,ij\notin T$) 
 but now it does not coincide with the center $\f z=\f{der}(\mcal O)\oplus\sum_{j\in I\setminus J}\mcal M_j$. 
 Recall that  any direct sum of a semisimple algebra and a solvable ideal $R,$ forces $R$ to coincide with the radical of that space. 
 This is exactly  our situation, since $\mcal L_T=\mcal L_0\oplus R$ with $\mcal L_0$ simple and $R$ an abelian ideal, due to
 $[R,\mcal L_0]\subseteq R$. Thus $\mcal L_{ss}=\mcal L_0$ and   $\f r= \oplus_{i\in I} \mcal L_i$. 
 Furthermore, we can provide more information on the lower and central series: 
     $[\mcal L_T,\mcal L_T]=\mcal L_0\oplus\sum_{j\in J}[\mcal L_0,\mcal L_j]=\mcal L_0\oplus\sum_{j\in J}\mcal M_j$, and then $[\mcal L_T,\mcal L_T]=\mcal L_T^{n}=\mcal L_T^{(n)}$ for all $n\ge2.$ 
    \smallskip
    
    Let us study how many isomorphism classes appear. The different Levi factors show that $\mcal L_{E_J}$ and $\mcal L_{F_K}$ cannot be  isomorphic. 
    Moreover, the dimensions of the centers imply that $|J| = |K|$ is a necessary condition for $\mcal L_{E_J}$ to be isomorphic to $\mcal L_{E_K}$ (and also for $\mcal L_{F_J}$ to be isomorphic to $\mcal L_{F_K}$).
    We can also provide an explicit isomorphism to verify that this condition is sufficient.
    If $J$ and $K$ are subsets of $I$ with the same cardinality, consider a bijection $\mu\colon I \to I$ such that $\mu(J) = K$ (not necessarily a collineation).  
Since $\dim \mcal D_i = 2$ for any $i \in I$, we can define a linear bijection $\theta_{ij}\colon \mcal D_i \to \mcal D_j$ for any $i,j \in I$.  
Now, define the linear isomorphism $f\colon \mcal L \to \mcal L$ by:
 $$
 \begin{array}{l}
      f\vert_{\mcal L_0}=\id , \\
 f\vert_{\mcal D_i}=\theta_{i\mu(i)},\\
 f(e_i\otimes u)=e_{\mu(i)}\otimes u, \quad \forall i\in I, \forall u\in \mcal J_0.
 \end{array}
 $$
Let us check that $f\colon \mcal L_{E_J}\to \mcal L_{E_K}$   and $f\colon \mcal L_{F_J}\to \mcal L_{F_K}$
are Lie algebra isomorphisms (note that $f\colon \mcal L \to \mcal L $ is not an isomorphism).
First consider these brackets in $\mcal L$:
\begin{itemize}
    \item If $x,y\in\mcal L_0$, then $f([x,y])=[x,y]=[f(x),f(y)]$;
    \item If $x \in\mcal L_0$, and $y=D+e_j\otimes u\in\mcal L_j$, then  $f([x,y])=e_{\mu(j)}\otimes x(u)=[f(x),f(y)]$;
     \item If $x  =D+e_j\otimes u\in\mcal L_j$, and $y  =D'+e_j\otimes v\in\mcal L_j$, then $f([x,y])=-4[R_u,R_v]=[f(x),f(y)]$;
\end{itemize}
though $f([x,y])\ne [f(x),f(y)]$ if $x\in\mcal L_j$ and $y\in\mcal L_k$, with $j\ne k$ in $I$ (the maps $\theta$ were chosen quite arbitrarily). 
 However, we can consider similar arguments when $T=E_J$ and  $T'=E_K,$ or  when $T=F_J$ and $T'=F_K.$ We have
$\ep^T_{jj}=\ep^{T'}_{\mu(j)\mu(j)}$, $\ep^T_{00}=\ep^{T'}_{00}$, and $\ep^T_{0j}=\ep^{T'}_{0\mu(j)}$ for all $j\in I$, (equal to either 0 or 1, depending on the chosen generalised nice set $T$). Thus, $f([x,y]^{\ep^{T}})= [f(x),f(y)]^{\ep^{T'}}$ whenever: $x,y\in\mcal L_0$, or $x \in\mcal L_0$, $y \in\mcal L_j$, or $x,y\in\mcal L_j$. Finally, if $x \in\mcal L_j$, $y \in\mcal L_k$, for $j\ne k$ in $I$, we have 
  $\ep^T_{jk}=0=\ep^{T'}_{\mu(j)\mu(k)}$ and then $f([x,y]^{\ep^{T}})= 0=[f(x),f(y)]^{\ep^{T'}}$.
\end{proof}

Before studying the generalised nice set $T = P_{\{0,j,j\}}$, let us first identify various semisimple Lie algebras that naturally appear as subalgebras of the exceptional Lie algebras $\mcal L = \mcal T(\mcal C)$.

 \begin{remark} If $\mcal J$
is a Jordan algebra, the \emph{structure algebra}, denoted by $\f{str}(\mcal J)$, is the Lie algebra generated by the multiplication operators. Taking into account that $[[R_u,R_v],R_w]=R_{(v,w,u)}$, then $\f{str}(\mcal J)=R_{\mcal J}+[R_{\mcal J},R_{\mcal J}]$. Likewise, the \emph{inner structure algebra}, denoted by $\f{innstr}(\mcal J)$, is the Lie algebra generated by the traceless multiplication operators, that is, $\f{innstr}(\mcal J)=R_{\mcal J_0}\oplus [R_{\mcal J},R_{\mcal J}]$. In our setting, that is,     ${\mathcal C}\in\{\mathcal F,\mathcal{K},\mathcal{H},\mathcal{O}  \}$  and  the Jordan algebra $\mcal J=\mathcal{H}_3({\mathcal C})$, the second row of the Freudenthal Magic Square gives its inner structure algebra, so that in particular,  $\f{innstr}(\mcal J)$ is semisimple. 

Also, the Lie
algebra in the third row  is the well-known Tits-Kantor-Koecher Lie algebra of the Jordan algebra
$\mcal J$ \cite{Tits62}.   
That is, $\f{tkk}(\mcal J)=\f{sl}_2\otimes \mcal J\oplus\f{der}(\mcal J)$, which is also 
semisimple in our case $\mcal J=\mathcal{H}_3({\mathcal C})$, as we see from Table~\ref{tab_magicSquare}.

\begin{table}[h]
    \centering
    \renewcommand{\arraystretch}{1.2} 
    \begin{tabular}{|c|cccc|} 
  \hline
  $\mathcal{C}$ 
        & $\mathcal{F}$ & $\mathcal{K}$ & $\mathcal{H}$ & $\mathcal{O}$ \\\hline\hline
        $\mathfrak{der}(\mathcal{J})$     & $\mathfrak{a}_1$ & $\mathfrak{a}_2$ & $\mathfrak{c}_3$ & $\mathfrak{f}_4$ \\
        $\mathfrak{innstr}(\mathcal{J})$  & $\mathfrak{a}_2$ & $2\mathfrak{a}_2$ & $\mathfrak{a}_5$ & $\mathfrak{e}_6$ \\
        $\mathfrak{tkk}(\mathcal{J})$     & $\mathfrak{c}_3$ & $\mathfrak{a}_5$ & $\mathfrak{d}_6$ & $\mathfrak{e}_7$ \\
        $\mathcal{T}(\mathcal{C})$        & $\mathfrak{f}_4$ & $\mathfrak{e}_6$ & $\mathfrak{e}_7$ & $\mathfrak{e}_8$ \\\hline
    \end{tabular}
    \vspace{0.5em}
    \caption{Freudenthal Magic Square}
    \label{tab_magicSquare}
\end{table}

   We would like to highlight that these algebras arise as distinguished graded subalgebras within our exceptional Lie algebras  $\mcal T(\mcal C)$.    More precisely, if $\Gamma_{\mcal L}$
     is our grading in Eq.~\eqref{eq_gradtodas}, then 
     $$
     \sum_{i\in I_0}\mcal L_i=\mcal T(\mcal C);\qquad
     \sum_{i=0,1,2,5}\mcal L_i\cong\f{tkk}(\mcal J);\qquad
     \sum_{i=0,1}\mcal L_i\cong\f{innstr}(\mcal J)\oplus2\f z;\qquad
      \mcal L_0=\f{der}(\mcal J).
     $$
    By $2\f z$ we mean a central ideal of dimension 2. In fact, we can find the inner structure algebra as $\mcal L_0\oplus \mcal M_1=\f{innstr}(\mcal J)$   (and so $\mcal{D}_1=2 \f z$). Note that the line $\{1,2,5\}$ can be replaced by any other.
    The reason for focusing on these subalgebras is that they will soon allow us to identify the Levi factors associated with certain graded contractions.
 \end{remark}
 
 Now that we know that \( \mathcal{L}_0 \oplus \mathcal{M}_j \) is semisimple for any \( j \in I \),  
the properties of the algebra \( \mathcal{L}_{P_{\{0,j,j\}}} \) follow immediately,  
since \(\mcal D_j\oplus \sum_{k\ne j}\mcal L_k \) is an abelian complement to it.

\begin{prop}\label{prop_elnoEniF}
For any \( j \in I \), let \( T = P_{\{0, j, j\}} \). Then the Lie algebra \( \mathcal{L}_T \) is reductive, with Levi factor  
\( \mathcal{L}_{ss} = \mathcal{L}_0 \oplus \mathcal{M}_j \cong \mathfrak{innstr}(\mathcal{J}) \),  
and radical \( \mathfrak{r} = \mathcal{D}_j \oplus \sum_{k \in I \setminus \{ j \}} \mathcal{L}_k = \mathfrak{z} \),  
which is an abelian Lie algebra.
\end{prop}

In particular, \( \mathcal{L}_{P_{\{0,j,j\}}} \) is not isomorphic to any of the Lie algebras encountered so far.  
We now turn to the study of Lie algebras arising from generalised nice sets \( T \) such that \( T \cap X \) is not itself a generalised nice set.  
To describe the sequences of dimensions of their lower central series, recall the notation \( \tilde \ell = \dim \mcal L_i = 3\ell + 4 \), where \( \ell = \dim \mcal C \).

\begin{prop}\label{pr_casolosYiexcepciones}
If $T$ is a generalised nice set as in Example~\ref{ex_las7excepciones}, then $\mcal L_T$ satisfies the following properties:
\begin{itemize}
    \item $T = X_0$. The algebra $\mcal L_T$ coincides with the original algebra $ \mcal L$, that is, it is simple.
    
    \item $T = Y_7$. The algebra is not reductive, with $\f z = \mcal L_3 \oplus \mcal L_4 \oplus \mcal L_6 \oplus \mcal L_7$ of dimension $4\tilde \ell$.
    The Levi decomposition is:
    \begin{itemize}
        \item $\mcal L_{ss} = \mcal L_0 \oplus \mcal M_1 \cong \f{innstr}(\mcal J)$;
        \item $\f r = \mcal D_1 \oplus \sum_{i \ne 0,1} \mcal L_i$ is 2-step solvable but not nilpotent. The dimensions of the derived series / lower central series of $\f r$ are given by the sequences $(6\tilde \ell + 2, 2\tilde \ell, 0)$ and $(6\tilde \ell + 2, 2\tilde \ell)$, respectively.
    \end{itemize}

    \item $T = Y_{11}$. The algebra is not reductive, with $\f z = \mcal L_4 \oplus \mcal L_7$ of dimension $2\tilde \ell$.
    The Levi decomposition is:
    \begin{itemize}
        \item $\mcal L_{ss} = \mcal L_0 \oplus \mcal M_1 \cong \f{innstr}(\mcal J)$;
        \item $\f r = \mcal D_1 \oplus \sum_{i \ne 0,1} \mcal L_i$ is 2-step solvable but not nilpotent. The dimensions of the derived series / lower central series of $\f r$ are given by the sequences $(6\tilde \ell + 2, 4\tilde \ell, 0)$ and $(6\tilde \ell + 2, 4\tilde \ell)$, respectively.
    \end{itemize}

    \item $T = Y_{15}$. The algebra is not reductive, with $\f z = 0$.
    The Levi decomposition is:
    \begin{itemize}
        \item $\mcal L_{ss} = \mcal L_0 \oplus \mcal M_1 \cong \f{innstr}(\mcal J)$;
        \item $\f r = \mcal D_1 \oplus \sum_{i \ne 0,1} \mcal L_i$ is 2-step solvable but not nilpotent. The dimensions of the derived series / lower central series of $\f r$ are given by the sequences $(6\tilde \ell + 2, 6\tilde \ell, 0)$ and $(6\tilde \ell + 2, 6\tilde \ell)$, respectively.
    \end{itemize}

    \item $T = Y_{19}$. The algebra is not reductive, with $\f z = 0$.
    The Levi decomposition is:
    \begin{itemize}
        \item $\mcal L_{ss} = \mcal L_0 \oplus \mcal M_1 \cong \f{innstr}(\mcal J)$;
        \item $\f r = \mcal D_1 \oplus \sum_{i \ne 0,1} \mcal L_i$ is 3-step solvable but not nilpotent. The dimensions of the derived series / lower central series of $\f r$ are given by the sequences $(6\tilde \ell + 2, 6\tilde \ell, 2\tilde \ell, 0)$ and $(6\tilde \ell + 2, 6\tilde \ell)$, respectively.
    \end{itemize}

    \item $T = Y_{10}$. The algebra is reductive.
    The Levi decomposition is:
    \begin{itemize}
        \item $\mcal L_{ss} = \sum_{i = 0,1,2,5} \mcal L_i \cong \f{tkk}(\mcal J)$;
        \item $\f z = \f r = \mcal L_3 \oplus \mcal L_4 \oplus \mcal L_6 \oplus \mcal L_7$ is abelian.
    \end{itemize}

    \item $T = Y_{26}$. The algebra is not reductive, with $\f z = 0$.
    The Levi decomposition is:
    \begin{itemize}
        \item $\mcal L_{ss} = \sum_{i = 0,1,2,5} \mcal L_i \cong \f{tkk}(\mcal J)$;
        \item $\f r = \mcal L_3 \oplus \mcal L_4 \oplus \mcal L_6 \oplus \mcal L_7$ is abelian.
    \end{itemize}
\end{itemize}
\end{prop}

 \begin{proof}
 Assume $T\in\{Y_{7},Y_{11},Y_{15},Y_{19}\}$. In these cases, $\ep^T_{00}=\ep^T_{01}=\ep^T_{11}=1$, so that the  semisimple subalgebra $L=\mcal L_0\oplus \mcal M_1\le\mcal T(\mcal C)$ ($L\cong \f{innstr}(\mcal J)$) is a subalgebra of $\mcal L_T$ too. 
 Take the complement $R=  \mcal D_1\oplus\sum_{i\ne1}\mcal L_i$  of $L$ and let us prove that $R$ coincides with the radical of $\mcal L_T$. 
 Taking into account   the facts
 $[\mcal D_1,\mcal L_j]=\mcal L_{1*j}$ if $j\ne1$, $[\mcal D_1,\mcal L_1]=0$, and $[\mcal L_i,\mcal L_j]=\mcal L_{i*j}$ if $i\ne j\ne1$, as well as the supports, we compute the bracket:  
\begin{equation}\label{eq_radicalYs}
[R,R]^{\ep^T}=\begin{cases}
    \mcal L_2\oplus \mcal L_5\quad\textrm{if $T=Y_{7}$,}\\
    \mcal L_2\oplus \mcal L_5\oplus \mcal L_3\oplus \mcal L_6\quad\textrm{if $T=Y_{11}$,}\\
     \oplus_{i\ne0,1} \mcal L_i\quad\textrm{if $T=Y_{15,19}$;}
\end{cases}
\ 
R^{(2)} =\begin{cases}
    0\quad\textrm{if $T=Y_{7,11,15}$,}\\
     \mcal L_4\oplus \mcal L_7\quad\textrm{if $T=Y_{19}$.}
\end{cases} 
\end{equation}
In particular, $R\le \mcal L_T$ is solvable, since $R^{(3)} =0$ for the 4 generalised nice sets. 
 Let us check that $R$ coincides with the radical $\f r(\mcal L_T)$: we simply need to check that $R$ is an ideal, that is, $[R,L]^{\ep^T}\subseteq R$. 
    This is straightforward: if $i\ne1$, $[\mcal L_0,\mcal L_i]\subseteq \mcal L_i$, 
    $[\mcal L_0,\mcal D_1]=0=[\mcal M_1,\mcal D_1]$,
    $[\mcal M_1,\mcal L_i]\subseteq \mcal L_{i*1}$, so $[R,L]\subseteq R$. If, instead,  we multiply with $[\,,\,]^{\ep^T}$, some of the previous brackets will vanish, but the containments will be maintained. 
 
Although all 4 algebras have common Levi decomposition, i.e. 
$\mcal L_{ss}=\mcal L_0\oplus \mcal M_1$ and $\f r= \mcal D_1\oplus\sum_{i\ne1}\mcal L_i$,
they are perfectly distinguishable according to \eqref{eq_radicalYs}. The radicals   have different dimension, except for 
   the latter two cases ($i=15$ and $19$), which have different solvability index.
   The four radicals are  not nilpotent,   since $R^n=R'$ for all $n\ge 1$. The reason is that, for each of the four generalised nice sets under consideration, the pair $1j \in T$ if and only if $\{1, 1 \ast j\} \in T$.  
\smallskip

 The remaining generalised nice sets ($Y_{10}$ and $Y_{26}$) both contain 
\(
Y_{10} = \{jk : j,k = 0,1,2,5\},
\)
which implies that the semisimple subalgebra 
\(
\sum_{i=0,1,2,5} \mcal L_i \cong \mathfrak{tkk}(\mcal J) \le \mcal T(\mcal C)
\)
is a subalgebra of $\mcal L_T$.  
In the two cases, the subspace 
\( R=\mcal L_3 \oplus \mcal L_4 \oplus \mcal L_6 \oplus \mcal L_7 \)
is abelian, since \( \{jk : j,k = 3,4,6,7\} \subseteq X_0 \setminus T \). 
Moreover, $R$ is an ideal, 
taking into account that, for any line $\ell\subseteq I$ in the Fano plane (e.g.\ \( \ell = \{1,2,5\} \)), we have 
\( \ell \ast \ell^c \subseteq \ell^c \), where \( \ell^c = I \setminus \ell \) is the complement. 
 In particular, 
$R$ coincides with the radical   \( \f r(\mcal L_T) \), while 
\( \sum_{i=0,1,2,5} \mcal L_i \)
is the Levi factor. 
In the case \( T = Y_{26} \), the algebra \( \mcal L_T \) has no center, since for \( j = 3,4,6,7 \), we have \( \mathfrak{z} \cap \mcal L_j = 0 \), due to \( 1j \in T \) and Lemma~\ref{le_comovanloscorchetes}(d).  
 \end{proof}
 
 Evidently, the algebras studied in the last proposition cannot be isomorphic to any of those that appeared earlier. Moreover, we may observe the following:

\begin{coro}
All the Lie algebras in the Freudenthal Magic Square appear as Levi factors of suitable graded contractions of \( \mcal L = \mathcal{T}(\mcal C) \).
\end{coro}

We have yet to analyze the cases in which the generalised nice set \( T \) is not contained in \( X \), but is not entirely outside \( X \) either.
Though the classification up to collineation in Theorem~\ref{teo_listaGNS} is complete,  
it is more convenient at this point to describe more conceptually---following the approach of \cite[Theorem~6.3]{draper2024generalised}---the generalised nice sets in \( X_0 \setminus X \)  
that can be joined to a generalised nice set contained in \( X \), rather than listing them case by case.

  \begin{prop}\label{pr_casogeneral}
 If a generalised nice set \( T \) satisfies neither \( T \subseteq X \) nor \( T \subseteq X_0 \setminus X \), then one of three situations   occurs:
 \begin{itemize}
     \item[\rm(a)] There are an index \( i = 1, \dots, 13 \), and a subset \( \emptyset \ne J \subseteq I \setminus J_i \) such that \( T = S_i \cup E_J \). 
     
     The related contracted Lie algebra  \( \mcal L_T \) is 2-step solvable and not reductive.
     \begin{itemize}
         \item[$-$] The lower central series is given by: \( \mcal L'_T = \mcal L_0 \oplus \sum_{k \in J_i} \mcal L_k \), and \( \mcal L^{2}_T = 0 \) if \( i \ne 13 \), and 
         \( \mcal L^{2}_T = \mcal L_4 \) if \( i = 13 \), with nilpotency index equal to 2 and 3, respectively.
         \item[$-$] The center equals \( \f z = \mcal L_0 \oplus \sum_{j \in I\setminus( J \cup K_i)} \mcal L_j \oplus \sum_{j \in J \setminus K_i} \mcal D_j \).
     \end{itemize}
     
     \item[\rm(b)] There is \( i = 1, \dots, 13 \), with \( T = S_i \cup F_J \) for some \( \emptyset \ne J \subseteq I \) such that, for any \( jk \in S_i \),
     either \( J \cap \{j, k, j*k\} = \emptyset \) or \( \{j, k, j*k\} \subseteq J \).

     The   algebra  \( \mcal L_T \) is not reductive, with center 
     \( \f z = \sum_{j \in I\setminus( J \cup K_i)} \mcal L_j \oplus \sum_{j \in J \setminus K_i} \mcal D_j \). The derived algebra is given by \( \mcal L'_T = \mcal L_0 \oplus \sum_{k \in J_i} \mcal L_k \oplus \sum_{k \in J \setminus J_i} \mcal M_k \). The Levi decomposition is given by:
     \begin{itemize}
         \item[$-$] The Levi factor is \( \mcal L_{ss} = \mcal L_0 \cong \f{der}(\mcal J) \);
         \item[$-$] The radical \( \f r = \f{der}(\mcal O) \oplus \mcal O_0 \otimes \mcal J_0 \) 
         is 2-step solvable, with \( [\f r, \f r] = \sum_{k \in J_i} \mcal L_k \), and nilpotency index equal to 2 or 3, according to whether \( i \ne 13 \) or \( i = 13 \).
     \end{itemize}
     
     \item[\rm(c)] There exist \( i = 1, \dots, 13 \), and \( j \in I \), with \( j \notin K_i \cup J_i \), such that \( T = S_i \cup P_{\{0,j,j\}} \).

Here, \( \mcal L_T \) is not reductive, with center  
\( \f z = \mcal D_j \oplus \sum_{k \notin K_i \cup \{0, j\}} \mcal L_k \). The Levi decomposition is given by:
\begin{itemize}
    \item[$-$] The Levi factor is \( \mcal L_{ss} = \mcal L_0 \oplus \mcal M_j \cong \f{innstr}(\mcal J) \);
    \item[$-$] The radical \( \f r = \mcal D_j \oplus \sum_{k \in I \setminus \{j\}} \mcal L_k \)  is again 2-step solvable, with \( [\f r, \f r] = \sum_{k \in J_i} \mcal L_k \).
\end{itemize}
  \end{itemize}
\end{prop}

  \begin{proof}
  As mentioned above, the conditions on \( j \) and \( J \) ensuring that the sets \( S_i \cup E_J \), \( S_i \cup F_J \), and \( S_i \cup P_{\{0,j,j\}} \) are generalised nice sets  
are extracted from~\cite[Theorem~6.3]{draper2024generalised}.
    Recall the sets $K_i$ and $J_i$ described in Table~\ref{tab:table_KiJi}, and that we use $j,k$ for indices in $I$, not in $I_0$.
    
Case (a). As in Proposition~\ref{pr_casolosSi}, we have \( \sum_{k \in J_i} \mcal L_k \subseteq \mcal L'_T \), but now there exists \( j \) such that \( jj \in T \), so
\( [\mcal L_j, \mcal L_j] = \mcal L_0 \subseteq \mcal L'_T \) as well. Clearly, there are no additional elements in the derived algebra.   Since \( 00, kk \notin T \) if $k\in J_i$, it follows that \( \mcal L_0\cap \mcal L_{T}^{(n)}=0$ for all \( n \ge 2 \). Also,  \( 0k \notin T \), so   that \( \mcal L_{T}^{(n)} = \mcal L_{S_i}^{(n)} \) and \( \mcal L_{T}^{n} = \mcal L_{S_i}^{n} \) for all \( n \ge 2 \), from which the solvability and nilpotency indices are derived.

Regarding the center, we have \( \f z \subseteq \f z(\mcal L_{S_i}) = \mcal L_0 \oplus \sum_{k \in I\setminus K_i} \mcal L_k \).
For any fixed \( k \notin K_i \), if \( k \notin J \), then \( \mcal L_k \subseteq \f z \); but if \( k \in J \), then \( \mcal L_k \cap \f z = \mcal D_k \), since \( [\mcal D_k, \mcal L_k] = 0 \) and \( \ep^T_{kj} = 0 \) for all \( j \ne k \).
\smallskip

    Case (b). The statement about the center is clear: if \( j \notin K_i \) and \( j \notin J \), then \( \mcal L_j \subseteq \f z \).
If \( j \notin K_i \) and \( j \in J \), then \( \ep^T_{jj} \ne 0= \ep^T_{jk}\) for all $k\ne j$, so \( \f z \cap \mcal L_j = \mcal D_j \).
Moreover, if \( j \in K_i \), then there exists \( k \in I \) such that \( jk \in T \), and thus \( \f z \cap \mcal L_j = 0 \) by Lemma~\ref{le_comovanloscorchetes}(d).

The computation of the derived algebra is also straightforward: since \( 00 \in T \), we have \( \f{der}(\mcal J) = \mcal L_0 \subseteq \mcal L'_T \).
Moreover, \( \sum_{k \in J_i} \mcal L_k \subseteq \mcal L'_T \) follows from the definition of \( J_i \) and the fact that \( [\mcal L_i, \mcal L_j] = \mcal L_{i*j} \) for 
distinct \( i, j \in I \). Finally, if \( j \in J \setminus  J_i \), then \( 0j \in T \), so
\(
[\mcal L_0, \mcal L_j]^{\ep^T} = [\mcal L_0, \mcal L_j] = \mcal M_j,
\)
and no nonzero element in \( \mcal D_j \) can appear as a bracket.

Since \( \f{der}(\mcal J)=\mcal L_0\le \mcal L_T \) is simple and \( \f{der}(\mcal O) \oplus \mcal O_0 \otimes \mcal J_0 \) is a solvable ideal of $\mcal L_T$, the Levi decomposition follows. The indices of \( \f r \) can be deduced from the case \( \mcal L_{S_i} \).
 \smallskip
  
Case (c). Note that \( P_{\{0,j,j\}} \subseteq T \) implies that \( \mcal L_0 \oplus \mcal L_j \) is a Lie subalgebra of \( \mcal L_T \), isomorphic to the direct sum of \( \f{innstr}(\mcal J) \) and a two-dimensional center (namely, \( \mcal D_j \)).

Let us check that the subspace \( R = \mcal D_j \oplus \sum_{k\ne j} \mcal L_k \)  is solvable. 
As $j\notin K_i$, \( [R, R]^{\ep^T} = [ \sum_{k \ne j} \mcal L_k, \sum_{k \ne j} \mcal L_k]^{\ep^T}    \). Thus
\([R, R]^{\ep^T} =\sum_{k \in J_i} \mcal L_k \), since any  $k\in J_i$ is of the form $k=k_1*k_2$ with $k_1,k_2\in K_i$, so that $k_1,k_2\ne j$ and $\mcal L_k=[\mcal L_{k_1},\mcal L_{k_2}]^{\ep^T}$.
Now, clearly \( R^{(2)} = 0 \).
Since \( [\mcal L_0 \oplus \mcal M_j, R] \subseteq R \), we have \( [\mcal L_0 \oplus \mcal M_j, R]^{\ep^T} \subseteq R \) as well, and hence \( R \) becomes the radical of $\mcal L_T$.

To compute the center, recall  again  that \( j \notin K_i \), so \( jj, j0 \in T \) and \( jk \notin T \) for all \( k \ne j \), which implies \( \f z \cap \mcal L_j = \mcal D_j \).
\end{proof}

It is not difficult to describe uniformly the action of \( \mcal L_{ss} \) on the radical.  
It coincides with the natural action of \( \f{der}(\mcal J) \), \( \f{innstr}(\mcal J) \), or \( \f{tkk}(\mcal J) \) on \( \f r \), depending on the chosen algebra \( \mcal C \).  
In all cases, the radical decomposes as a direct sum of several copies of the trivial module and several copies of the module \( \mcal J \), depending on the support.
\smallskip

Until this point, we have described the properties of the algebras \( (\mathcal{T}(\mathcal{C}))^{\varepsilon^T} \) for each generalised nice set \( T \)--- moreover, these properties coincide with those of \( (\mathcal{T}(\mathcal{C}))^{\varepsilon} \) for any other graded contraction \( \varepsilon \) with the same support \( T \).  
A second and important question is how many non-isomorphic algebras arise within this family.

The fact that, in Theorem~\ref{teo_listaGNS}, all the described sets are non-collinear does not imply that they give rise to non-isomorphic algebras. We write \( T \simeq T' \) if \( \varepsilon^T \sim \varepsilon^{T'} \), that is, if there exists an isomorphism of \( \mathbb{Z}_2^3 \)-graded Lie algebras \( f \colon \mathcal{L}_T \to \mathcal{L}_{T'} \). 
  (A priori, this notion may depend on the Hurwitz algebra involved.)   
 The condition means that there exists a map \( \sigma \colon I_0 \to I_0 \) such that \( f(\mathcal{L}_i) = \mathcal{L}_{\sigma(i)} \) for every \( i \in I_0 \). 
  Clearly, \( \sigma(0) = 0 \) and \( \sigma\vert_I \) is a bijection, with \( T' = \tilde\sigma(T) \), where we define \( \tilde\sigma(ij) = \sigma(i)\sigma(j) \). 
It seems natural to believe that non-collinear supports yield distinct algebras because,
at least if $jk\in T$,  we prove below that $\sigma(j)*\sigma(k)=\sigma(j*k)$. 
(Recall that the converse is true by Eq.~\eqref{eq_extendiendo_colin_a_iso}, if \( \sigma \in S_\ast(I) \) satisfies \( T' = \tilde\sigma(T) \), then there exists an isomorphism \( f \colon \mcal L_T \to \mcal L_{T'} \) with associated map \( \sigma \).) However, as a     consequence of the proof   
of Proposition~\ref{pr_casolosSi}, the map \( \sigma \) associated with an isomorphism \( f \) as above does not necessarily have to be a collineation.
To be precise, we proved that \( \ep^{S_4} \sim \ep^{S_5} \), although \( S_4 \) and \( S_5 \) are not collinear. The same situation occurred when we classified the graded contractions of \( \Gamma_{\mathfrak{g}_2} \). In that setting, this was the only exception, so that we obtained 23 non-isomorphic algebras out of 24 nice sets.
The natural question that arises at this point is: how many distinct Lie algebras result from our 245 non-collinear generalised nice sets?  
We will address this question in Theorem~\ref{teo_representantes}.  
It turns out that many of the resulting algebras are non-isomorphic, but we should be more precise.

 Many of the cases will be studied in a fairly unified way, making use of the following technical lemma.
 
 \begin{lemma}\label{le_util paranoisomorfas} 
For each  generalised nice set $T$ and $i \in I_0$,  let us denote by $n_{i,T} = |\{j \in I : ij \in T\}|$.

If $f\colon \mcal L_{T} \to \mcal L_{T'}$ is a graded Lie algebra isomorphism for two generalised nice sets $T$ and $T'$, and $\sigma\colon I_0 \to I_0$ is the bijection such that $f(\mcal L_{i}) = \mcal L_{\sigma(i)}$, then:
\begin{itemize}
    \item[\rm(i)] For any $i,j \in I_0$, we have $ij \in T$ if and only if $\sigma(i)\sigma(j) \in T'$.
    
    \item[\rm(ii)] If $ij \in T$, then $\sigma(i*j) = \sigma(i) * \sigma(j)$.
    
    \item[\rm(iii)] $\sigma(0) = 0$; $\tilde\sigma(E_I) = E_I$; $\tilde\sigma(F_I) = F_I$; and $\tilde\sigma(T \cap X) = T' \cap X$.
    
    \item[\rm(iv)] $T \subseteq X$ if and only if $T' \subseteq X$; \ 
    $T \subseteq X_0 \setminus X$ if and only if $T' \subseteq X_0 \setminus X$;
  \  $T \cap X$ is not generalised nice if and only if $T' \cap X$ is not generalised nice;
    \  and $T \cap (X_0 \setminus X) \subseteq E_I$ (respectively, contained in $F_I$ or in $P_{\{0,j,j\}}$) if and only if the same holds for $T'$.
    
    \item[\rm(v)] For any $i \in I_0$, we have $n_{i,T} = n_{\sigma(i),T'}$.
    
    \item[\rm(vi)] If $T \cap X = S_i$ and $T' \cap X = S_j$ for $i,j \in \{0,1,\dots,13\} \setminus \{5\}$, then $i = j$.
\end{itemize}
\end{lemma}

 \begin{proof} (i) For any $i,j\in I_0$,   there are  $x\in\mcal L_i$ and $y\in\mcal L_j$ such that $[x,y]\ne0$.
 If $ij\in T$, then $[x,y]^{\ep^T}=[x,y]\ne0$, also $0\ne f([x,y]^{\ep^T})=[f(x),f(y)]^{\ep^{T'}}$. 
 As $f(x)\in\mcal L_{\sigma(i)}$ and $f(y)\in\mcal L_{\sigma(j)}$, 
 this means that $\ep^{T'}_{\sigma(i)\sigma(j)}\ne0$, in other words,
 $ \sigma(i)\sigma(j)\in T'$. We have proved more: since 
 $[f(x),f(y)]^{\ep^{T'}}\in \mcal L_{ \sigma(i)*\sigma(j)}$ equals $f([x,y]^{\ep^T})\in\mcal L_{\sigma(i*j)}$ and is nonzero, then 
 $\sigma(i*j)=\sigma(i)*\sigma(j)$, and (ii) follows too. (Note that this does not mean that $\sigma$ is a collineation.)

    (iii) It is clear that $\sigma(0)=0$ since $\dim\mcal L_0\ne \dim\mcal L_i$ if $i\ne0$. 
    The remaining statements are then evident.

    (iv) We want to prove that \( T \) and \( T' \) belong to the same class of generalised nice sets, according to the description in Section~\ref{se_GNS}.  
Since \( \tilde{\sigma} \) preserves both \( X \) and \( X_0 \setminus X \), the first two cases are immediate.  
If \( T \) is contained in neither \( X \) nor \( X_0 \setminus X \), and \( T \cap X \) is not generalised nice, then either \( T \cap X \) is not a nice set,  
or there exist \( i \ne j \) in \( I \) such that \( \{i,j\}, \{i, i * j\} \in T \), by \cite[Proposition~3.1]{draper2024generalised}.  
Now, by (i) and~(ii), it follows that either \( T' \cap X \) is not a nice set,  
or \( \{\sigma(i), \sigma(j)\}, \{\sigma(i), \sigma(i) * \sigma(j)\} \in T' \), so \( T' \cap X \) is not generalised nice either.  
Item~(iii) immediately covers the last case.

    (v) is clear from (i).

    (vi) Simply note that the sequences \( N_i = (n_{k, S_i} : k \in K_i) \), ordered from largest to smallest, are all different.  
This is evident for the sets \( S_i \) of different cardinalities---the sum of the entries in the sequence \( N_i \) equals \( 2|S_i| \)---and very easy to check for those with the same cardinality:
\begin{itemize}
    \item (Cardinality 2) \( N_2 = (2,1,1) \) and \( N_3 = (1,1,1,1) \);
    \item (Cardinality 3) \( N_4 = (3,1,1,1) \), \( N_6 = (2,2,2) \), \( N_7 = (2,2,1,1) \), and \( N_8 = (1,1,1,1,1,1) \);
    \item (Cardinality 4) \( N_9 = (3,2,2,1) \) and \( N_{10} = (2,2,2,2) \);
    \item (Cardinality 6) \( N_{12} = (3,3,3,3) \) and \( N_{13} = (3,3,3,1,1,1) \).
\end{itemize}
 \end{proof}

From this point, a careful analysis allows us to determine how many non-isomorphic Lie algebras can be obtained by applying the generic graded contractions \( \varepsilon^T \) to the grading \( \Gamma_{\mathcal{L}} \), as \( T \) ranges over all generalised nice sets.

 \begin{theorem}\label{teo_representantes}
 
Let \( \mathcal{C} \)   be a fixed Hurwitz algebra, and let \( \mathcal{L} = \mathcal{T}(\mathcal{C}) \). Then there are exactly \( 215 \) non-isomorphic Lie algebras in the set
\[
\{ \mathcal{L}_T : T \text{ is a generalised nice set} \}.
\]
A list of representatives is given by 
      \begin{enumerate} 
         \item $Y_{7}$, $Y_{10}$, $Y_{11}$, $Y_{15}$, $Y_{19}$, $Y_{26}$,   and $X_0$;
         \item $  P_{\{0,1,1\}}$, $S_1\cup P_{\{0,3,3\}}$, $S_2\cup P_{\{0,4,4\}}$,
           $S_3\cup P_{\{0,3,3\}}$, $S_6\cup P_{\{0,7,7\}}$,
         $S_7\cup P_{\{0,4,4\}}$, and $S_{10}\cup P_{\{0,4,4\}}$;
         
         \item $E_J$ and $F_J$ for $J=\emptyset,1,12,123, 1234, 12345,123456,I$;
         \item $S_i\cup E_J$ for
         \begin{itemize}  
             \item[-] $i=1$, $J=\emptyset,1,3,12,13,34, 123,134,346, 1234,3467, 1367,12346,13467,123467 $;
             \item[-] $i=2 $, $J=\emptyset, 1,2,4, 12,14, 23,24, 47,123,124, 234,247,147, 1234, 1247,2347,12347$;
              \item[-] $i=3 $, $J=\emptyset, 1,3,12,13,16,34,123,126,134, 136,1234,1267,1236,1346,12346,\\12367,123467 $;
               \item[-] $i= 4$, $J=\emptyset, 1,2,12,23,123,234,1234$;
                 \item[-] $i=6 $, $J=\emptyset,1,7,12,17,126,127,1267 $;
                  \item[-] $i= 7$, $J=\emptyset,1,2,4,12,14,16,17,24,27,124,126,127,146,247,147,1267,1246,\\1247,12467 $;
                   \item[-] $i= 8$, $J=\emptyset, 2,23,25,234,235, 2356,2345,23456,234567$;
                    \item[-] $i= 9$, $J=\emptyset,1,2,7,12,17,26,27,126,127,267,1267 $;
                     \item[-] $i=10 $, $J=\emptyset,1,4,12,14,17,124,126,147,1267,1246,12467 $;
                      \item[-] $i=11 $, $J=\emptyset,1,6,12,16,67,126,167,1267 $;
                       \item[-] $i=12 $, $J=\emptyset, 3,34,346,3467$;
                        \item[-] $i= 13$, $J=\emptyset, 1,12,123$.
         \end{itemize}
         \item $S_i\cup F_J$ for
         \begin{itemize} 
          \item[-] $i=1$, $J=\emptyset, 3,34, 125,347,3467,1235,12567, 123456,I$;
           \item[-] $i=2$, $J= \emptyset,4, 47,12356,123456, I$;
            \item[-] $i=3$, $J=\emptyset, 3,34,12567,123567,I$;
             \item[-] $i=4, 8,9,11,12,13$, $J=\emptyset, I$;
              \item[-] $i=6$, $J= \emptyset,7,123456,I$;
              \item[-] $i=7,10$, $J= \emptyset,4,123567,I$.
         \end{itemize}
     \end{enumerate}
 \end{theorem}

 \begin{proof}
  Recalling Theorem~\ref{teo_listaGNS}, it will suffice to check each of the following:
\begin{itemize}
    \item[\rm(i)]  $S_2\cup P_{\{0,7,7\}}\simeq S_2\cup P_{\{0,4,4\}}$;
    \item[\rm(ii)]  $E_{125}\simeq E_{123}$, $E_{1235}\simeq E_{1234}$, $F_{125}\simeq F_{123}$ and $F_{1235}\simeq F_{1234}$;
    \item[\rm(iii)]  $S_1\cup E_{36}\simeq S_1\cup E_{34}$, $S_1\cup E_{136}\simeq S_1\cup E_{134}$ and $S_1\cup E_{1236}\simeq S_1\cup E_{1234}$;
    \item[\rm(iv)]  $S_2\cup E_{7}\simeq S_2\cup E_{4}$, $S_2\cup E_{17}\simeq S_2\cup E_{14}$, 
    $S_2\cup E_{27}\simeq S_2\cup E_{24}$, 
    $S_2\cup E_{127}\simeq S_2\cup E_{124}$, $S_2\cup E_{237}\simeq S_2\cup E_{234}$ and $S_2\cup E_{1237}\simeq S_2\cup E_{1234}$;
    \item[\rm(v)]  $S_3\cup E_{137}\simeq S_3\cup E_{136}$;
    \item[\rm(vi)]  $S_4\cup E_{J}\simeq S_5\cup E_{J'}$ where $J'$   is equal to $J$ except that each $4\in J$ is replaced by a 7 in $J'$, for $J$ 
    any of the subsets in Theorem~\ref{teo_listaGNS}(4);
    \item[\rm(vii)]  $S_8\cup E_{237}\simeq S_8\cup E_{234}$;
    \item[\rm(viii)]  $S_1\cup F_{37}\simeq S_1\cup F_{34}$ and $S_1\cup F_{12356}\simeq S_1\cup F_{12567}$;
    \item[\rm(ix)]  $S_2\cup F_{7}\simeq S_2\cup F_{4}$ and $S_2\cup F_{123567}\simeq S_2\cup F_{123456}$;
    \item[\rm(x)]    $S_4\cup F_\emptyset \simeq S_5 \cup F_\emptyset$    and $S_4\cup F_I\simeq S_5\cup F_I$;
    \item[\rm(xi)]  None of the remaining generalised nice sets could give   rise to isomorphic Lie algebras.
\end{itemize}

The arguments are all similar, so we present only those that introduce something new.  
For any \( i, j \in I \), we can find a linear bijection \( \theta_{ij} \colon \mathcal{D}_i \to \mathcal{D}_j \).  
We extend \( \theta_{ij} \) to a linear map \( \hat\theta_{ij} \colon \mathcal{L}_i \to \mathcal{L}_j \)  
by defining \( \hat\theta_{ij}|_{\mathcal{D}_i} = \theta_{ij} \), and  
\( \hat\theta_{ij}(e_i \otimes u) = e_j \otimes u \) for any \( u \in \mathcal{J}_0 \).

  For (i), recall that $T=\{12,13,00,04,44\}$ is not collinear to $T'=\{12,13,00,07,77\}$. 
 Now simply take the linear isomorphism $f\colon \mcal L \to \mcal L $ given by
 \begin{equation}\label{eq_iso47}
   \begin{array}{l}
      f\vert_{\mcal L_i}=\id \qquad \forall i\in\{0,1,2,3,5,6\}, \\
      f\vert_{\mcal L_4}=\hat\theta_{47},\qquad
 f\vert_{\mcal L_7}=\hat\theta_{74}.\\
 \end{array}  
 \end{equation} 
 Although  $f\colon \mcal L \to \mcal L$ is not an isomorphism of Lie algebras, it is easy to show that $f\colon \mcal L_{T}\to \mcal L_{T'}$ is a Lie algebra isomorphism. 
 Indeed, it suffices to check the action of \( f \) on the brackets \( [\mathcal{L}_0, \mathcal{L}_4] \) and \( [\mathcal{L}_4, \mathcal{L}_4] \), using Eq.~\eqref{eq_TitsProduct}.
 Notice first that $[\mcal L_0,\mcal D_4]=0=[\mcal L_0,\mcal D_7]$. 
 Second, for $D\in\mcal L_0$ and $u\in\mcal J_0$, then 
 $f([D,e_4\otimes u])=e_7\otimes D(u)=[f(D),f(e_4\otimes u)]$.
 Third, since $D_{e_i,e_i}=0$ and $[e_i,e_i]=0$ for any $i\in I$, then $[e_i\otimes u,e_i\otimes v]=-4[R_u,R_v]$
 and then  $f([e_4\otimes u,e_4\otimes v])=-f(4[R_u,R_v])=-4[R_u,R_v]=[e_7\otimes u,e_7\otimes v]=[f(e_4\otimes u),f(e_4\otimes v)]$. Fourth, recall that $[\mcal D_i,\mcal M_i]=0$ for both $i=4,7$. And finally $[\mcal D_i,\mcal D_i]=0$ too.

 For (ii), let us check that $E_J\simeq E_{J'}$ whenever  $|J|=|J'|$. Take $\sigma\colon I\to I$ any bijection such that $\sigma(J)=J'$.  
  Now define $f\vert_{\mcal L_0}=\id $, 
 $f\vert_{\mcal L_i}=\hat\theta_{i\sigma(i)}$ for any $i\in I$.
  Straightforward computation reveals that $f\colon \mcal L_{E_{J}}\to \mcal L_{E_{J'}}$ is an isomorphism of graded algebras: As $[\mcal D_i,\mcal D_i]=0=[\mcal D_i,\mcal M_i]$, we  need only check the restrictions $f\vert_{[\mcal M_i, \mcal M_i]}$ if $i\in J$, but recall that $[e_i\otimes u,e_i\otimes v]=-4[R_u,R_v]$ does not depend on $i$ and $f(-4[R_u,R_v])=-4[R_u,R_v]$.   The same map $f$ is also an isomorphism $f\colon \mcal L_{F_{J}}\to \mcal L_{F_{J'}}$  and so we conclude that $F_J\simeq F_{J'}$ whenever  $|J|=|J'|$.

 Similarly $f\colon \mcal L_{T}\to \mcal L_{T'}$ defined by 
 $f\vert_{\mcal L_i}=\id  $ for all $ i\in\{0,1,2,3,5,7\}$,  
      $f\vert_{\mcal L_4}=\hat\theta_{46}$, 
 $f\vert_{\mcal L_6}=\hat\theta_{64}$, is an isomorphism for any of the pairs $(T,T')$ in item (iii).
  Again, $f\colon \mcal L_{T}\to \mcal L_{T'}$ defined by 
 $f\vert_{\mcal L_i}=\id  $ for all $ i\in\{0,1,2,3,5,6\}$,  
      $f\vert_{\mcal L_4}=\hat\theta_{47}$, 
 $f\vert_{\mcal L_7}=\hat\theta_{74}$, is an isomorphism for any of the pairs $(T,T')$ in item (iv).
  In both (iii) and (iv), $T=S_i\cup E_J$ and $f\vert_{\mcal L_k}=\id $ for any $k\in J_i\cup K_i$. 
  Therefore, all the remaining checks
 concern $E_J$ and they work   much as in the two previous paragraphs.
 
 Note that the isomorphism \( f \colon \mathcal{L}_{E_J} \to \mathcal{L}_{E_{J'}} \), defined as \( \tilde{\varphi} \) in the proof of Proposition~\ref{pr_casolosSi}, is valid for all the cases in~(vi) and~(x).
 A minor modification  of $f$ yields the required isomorphism in several additional cases.
    Fix $\ell=(i,j,i*j)$ as one of the ordered lines of the Fano plane and $l\in \ell^c=I\setminus\ell$.   We can find three derivations $x_k\in\f{der}(\mathcal{O})$ determined by $x_k(e_i)=0=x_k(e_j)$ and $x_k(e_l)=\frac12e_ke_l$ for $k\in\ell$, and   another three by $y_{i_1}=\frac14D_{e_{i_2},e_{i_3}}$, for $(i_1,i_2,i_3)$ any of the three ordered permutations of $\ell$. The map 
 $$
    \begin{array}{ccc}
    x_{i*j}\mapsto x_{i*j},  & x_i\mapsto x_j,& x_j\mapsto -x_i,\\
     y_{i*j}\mapsto y_{i*j}, &  y_i\mapsto y_j,& y_j\mapsto -y_i,
     \end{array}
    $$
     is a Lie algebra automorphism of $\mcal D_{i}\oplus \mcal D_{j}\oplus \mcal D_{i*j} $. 
     We extend this to a linear isomorphism \( \varphi_{ij} \colon \mcal D \to \mcal D \) by setting \( \varphi_{ij}\vert_{\mcal D_k} = \id \) for all \( k \notin \ell \), and then further extend it to a linear bijection \( \hat{\varphi}_{ij} \colon \mcal L \to \mcal L \) by defining \( \hat{\varphi}_{ij}\vert_{\f{der}(\mcal J)} = \id \), and for all \( u \in \mcal J_0 \), 
    $$
     e_i\otimes u\mapsto e_j\otimes u,\quad
     e_j\otimes u\mapsto -e_i\otimes u,\quad
 e_k\otimes u\mapsto e_k\otimes u \quad\forall k\in I\setminus\{i,j\}.
    $$
Now, it is straightforward to check that $\hat\varphi_{67}$ provides the required isomorphism in (v).
Also, for $S_8\cup E_{237}\simeq S_8\cup E_{234}$ in (vii) we consider $\hat\varphi_{74}$. 
For $S_1\cup F_{37}\simeq S_1\cup F_{34}$ we can use again the map in \eqref{eq_iso47}, as well for the two cases in (ix). 
The other case in (viii) uses a similar map to \eqref{eq_iso47} but involving the indices \( 3 \) and \( 7 \).
\medskip

Therefore,   it only remains to  prove (xi), that is, the remaining generalised nice sets  cannot give rise to isomorphic Lie algebras.   Item (iv) of 
Lemma~\ref{le_util paranoisomorfas} says that  generalised nice sets  $T$ and $T'$ which occur in different lists (1), (2), (3), (4), or (5) in Theorem \ref{teo_representantes},  
  cannot give equivalent graded contractions. We therefore consider each list separately.
  
(1) Each of these generalised nice sets have different cardinalities.
 In particular they give 7 non-isomorphic Lie algebras, by Lemma~\ref{le_util paranoisomorfas}.

(2) In this list there is nothing to check by Lemma~\ref{le_util paranoisomorfas}(vi).

(3) It is clear, because the  cardinalities of the subsets $J$ all differ.

  (4) As above, we can consider separately each case  $S_i\cup E_{J}$ for fixed $i\ne 5$. 
   If there is $f\colon \mcal L_{S_i\cup E_{J}}\to \mcal L_{S_i\cup E_{J'}}$ with 
  $f(\mcal L_{i})=\mcal L_{\sigma(i)}$  for a bijection
  $\sigma\colon I_0\to I_0$, then $\tilde\sigma(S_i)=S_i$ and $\sigma( J)=J'$. In particular,  $\sigma( K_i\cap J)= K_i\cap J'$, so that not only $|J|=|J'|$ but also $|K_i\cap J|=|K_i\cap J'|$. 
 
  \begin{itemize}
      \item[$-$] $i=1$. The two sets with $|J|=1$, give arise to  non-equivalent graded contractions $\ep^{ \{12,11 \}}$  and $\ep^{  \{12, 33\}}$, since  $ K_1\cap J=\{1,2\}\cap J$ has $1$ and $0$ elements, respectively.   There are three possible $J$'s with cardinality 2: $\{1,2\}$, $\{1,3\}$, $\{3,4\}$ so that $| K_1\cap J|=2,1,0$ respectively. Thus, again the related graded contractions $\ep^{ S_1\cup E_{J}}$ cannot be equivalent.
      Similarly, besides the set $J$ of cardinal six, we have
      
      \begin{table}[h]
          \centering
          \begin{tabular}{c||cccccccc}
              $J$ & 123 & 134 & 346 & 1234 & 3467 & 1367 & 12346 & 13467\\
              \hline 
             $(|J|,| K_1\cap J|)$  & $(3 ,2 )$ & $(3 , 1)$ & $( 3,0 )$ & $( 4,2 )$ & $( 4, 0)$ & $( 4, 1)$ & $( 5, 2)$ & $(5 ,1 )$\\
          \end{tabular}
      \end{table}
    \noindent   Hence all the related Lie algebras are non-isomorphic. \smallskip

       \item[$-$] $i=2$. In this case, the arguments from \( i=1 \) need to be slightly adapted, since 
       $K_2=\{1,2,3\}$  satisfies $| K_2\cap J|=1$ for both $J=\{1\}$ and $\{2\}$. In any case, note that the indices $1$ and $2$ play different roles, 
       since \( 1 \) appears more frequently than \( 2 \) in the elements of \( S_2 \).
      Thus, we use the property whereby $\sigma $ preserves  the subsets of indices with a fixed frequency in $S_2$. That is, $\sigma$ preserves $\{1\}$, $\{2,3\}$ and $\{4,5,6,7\}$ and hence the cardinals of their intersections with $J$. These are given by:

       \begin{table}[h]
           \centering
           \begin{tabular}{cccccccc}
               1 &  2& 4 & 12 & 14 & 23 & 24 & 47 \\
               (1,0,0) & (0,1,0) & (0,0,1) & (1,1,0) & (1,0,1) & (0,2,0) & (0,1,1) & (0,0,2) \\
               \hline
                 123 & 124  & 234 & 247 & 147 & 1234 & 1247&2347\\
                  (1,2,0) & (1,1,1)  & (0,2,1) & (0,1,2) & (1,0,2) & (1,2,1) & (1,1,2) & (0,2,2) \\
           \end{tabular}
       \end{table}
              As they are all different (different  from the last $J$ with 5 elements too), the related algebras cannot be isomorphic.
              \smallskip
              
               \item[$-$] $i=3$.   Here there are different pairs of $J$'s 
                              with the same invariant $(|J|,|K_3\cap J|)$, so that we must refine the argument.
                              This is the case for $J=12,16$, for $J=123,136$ and for $J=1234, 1346$. 
               We claim that $S_3\cup E_{12}\simeq S_3\cup E_{16}$ is not possible, since the bijection $\sigma$ related to such an isomorphism 
               would preserve  $J_3=\{5\}$ too, and so $\sigma(\{1,2\}) \neq \{1,6\}$ (recall $1*2=5\ne1*6$). 
               The cases with \( |J| \in \{3,4\} \) follow by similar arguments, using again  that, if an isomorphism existed, the corresponding \( \sigma \) would preserve both \( K_3 = \{1,2,6,7\} \) and \( \{5\} \), so that \( \sigma(\{1,2\}) \neq \{1,6\} \).
\smallskip

\item[$-$] $i=4$.   Notice that $\sigma$ must preserve $K_4=\{1,2,3,4\}$ and also $\sigma(1)=1$, which is the only index with $n_{1,S_4}=3$. Thus, the cases $J=\emptyset, 1,2,12,23,123,234,1234$, are all differentiated by their respective invariants $(|J\cap\{1\}|,|J\cap\{2,3,4\}|)=(0,0),$ $(1,0),(0,1),(1,1),(0,2),(1,2),(0,3),(1,3)$. \smallskip

\item[$-$] $i=6$. 
This case is not difficult, since the cardinalities \( |K_6 \cap J| \) for \( J = 1, 7 \) are \( 1 \) and \( 0 \);  
for \( J = 12, 17 \) they are \( 2 \) and \( 1 \); and for \( J = 126, 127 \), they are \( 3 \) and \( 2 \), respectively.\smallskip

\item[$-$] $i=7$. 
Despite the many possible choices of \( J \), the sequence \( (|J|,\, |J \cap \{1,6\}|,\, |J \cap \{2,7\}|) \) distinguishes all cases, except for \( J = 12 \) and \( 17 \) (both associated with the triple \( (2,1,1) \)), and \( J = 124 \) and \( 147 \) (both associated with \( (3,1,1) \)). Moreover, any bijection \( \sigma \) as in the previous cases would satisfy \( \tilde{\sigma}(12) \in \tilde{\sigma}(S_7) = S_7 = \{12,16,67\} \), so that  
\( \sigma(\{1,2\}) \neq \{1,7\} \). Thus, \( S_7 \cup E_{12} \not\simeq S_7 \cup E_{17} \). Similarly, \( S_7 \cup E_{124} \not\simeq S_7 \cup E_{147} \), since \( \sigma(K_7) = K_7 = \{1,2,6,7\} \). \smallskip

\item[$-$] $i=8$.   Again, any appropriate bijection \( \sigma \) must preserve \( S_8 \).  
In particular, \( \sigma(23) \neq 25 \), \( \sigma(234) \neq 235 \), and  
\( \sigma(2356) \neq 2345 \), by considering the respective quantities  
\( |J \cap \{2,5\}| \), \( |J \cap \{3,6\}| \), and \( |J \cap \{4,7\}| \),  
which should match (up to reordering).  
However, in the mentioned sets \( J \), these values are \( 110 \) vs. \( 200 \); \( 111 \) vs. \( 210 \); and \( 220 \) vs. \( 211 \), respectively.\smallskip

\item[$-$] \( i = 9 \). Taking into account the frequencies of \( 1, 2, 6 \), and \( 7 \) in \( S_9 \), a simple inspection of the cardinalities  
\( (|J|,\, |J \cap \{1\}|,\, |J \cap \{2,6\}|,\, |J \cap \{7\}|) \) reveals that all the generalised nice sets correspond to non-equivalent graded contractions.\smallskip

\item[$-$] $i=10$. This case is similar to $i=8$, in the sense that the pair of cardinals $(|J|, |J\cap\{1,2,6,7\}|)$ fails  only to distinguish $12$ from $17$, and $124$ from $147$. But $\tilde\sigma(S_{10})=S_{10}=\{12,16,27,67\}$   forces $\sigma(\{1,2\})\ne \{1,7\}.$\smallskip

\item[$-$] $i=11$. 
Any suitable bijection \( \sigma \) must preserve the sets \( \{1,2\} \) and \( \{6,7\} \), since \( n_{1,S_{11}} = n_{2,S_{11}} = 3 \) and \( n_{6,S_{11}} = n_{7,S_{11}} = 2 \).  
We can distinguish between each possible \( J \) by considering the sequence \( (|J|,\ |J \cap \{1,2\}|,\ |J \cap \{6,7\}|) \).\smallskip

\item[$-$] $i=12,13$.   Trivial, since all the possibilities for $J$ have different   cardinals.
  \end{itemize}

(5) The situation for \( S_i \cup F_J \) is much simpler, since the cardinality of \( J \) distinguishes between most cases.  
The only cases not distinguished by cardinality occur when \( S_1 = \{12\} \), and we ask whether  
\( S_1 \cup F_{125} \simeq S_1 \cup F_{347} \) and whether \( S_1 \cup F_{1235} \simeq S_1 \cup F_{3467} \).  
These cases are ruled out because any suitable \( \sigma \) must preserve \( J_1 = \{5\} \).
 \end{proof}

\section{Conclusion}

We have found a total of \(860\) Lie algebras, \(215\) for each dimension in the set \(\{52,\linebreak[1] 78,\linebreak[1] 133,\linebreak[1] 248\}\), with explicit representatives given in Theorem~\ref{teo_representantes}.  
Their properties are described in Proposition~\ref{pr_casolosSi}, Proposition~\ref{pr_losEsyFs}, Proposition~\ref{prop_elnoEniF}, Proposition~\ref{pr_casolosYiexcepciones}, and Proposition~\ref{pr_casogeneral}.  
Remarkably, these propositions apply equally to any other graded contraction of \( \mcal T(\mathcal{C}) \), since the results depend only on the support.

It may be tempting to think that, since our graded contractions are generic and can be applied to any other $\mathbb Z_2^3$-graded Lie algebra, the resulting properties and structural results might be similar in general. However, this is not the case. For instance, the situation for \( \mathfrak{g}_2 \) already differs significantly (see~\cite{draper2024gradedg2}). Moreover, the proof of Theorem~\ref{teo_representantes}   relies on the specific properties of the $G$-gradings used in this work.

Further work could explore the existence of additional non-isomorphic Lie algebras obtained by graded contractions of $\Gamma_{\mcal T(\mathcal{C}) }$, with the same support as those described in Theorem~\ref{teo_representantes}.  
It seems unlikely that \( \varepsilon^T \) is the unique graded contraction, up to equivalence, with support \( T \) (for fixed \( T \)); although for some choices of \( T \), this could indeed be the case.  
In particular, it appears that one could choose images different from \( \{0,1\} \).  
The study of 2-coboundaries can significantly reduce the number of cases to consider, often leaving only a few equivalence classes associated with each fixed support.  
However, this remains a non-trivial problem, which has been addressed, for instance, in~\cite{draper2024gradedg2}.

 That is, the work could easily be continued; in any case, it is astonishing to have achieved such a complete classification of the graded contractions of Lie algebras of dimension \( 248 \)!---even more so considering that it was done without the aid of a computer.  
Moreover, we are not aware of any other example of such an extensive collection of graded contractions for any family of Lie algebras.  
As an additional merit of this work, let us emphasize the simplicity with which the properties of the resulting algebras can be described, once the appropriate approach to study them was found.



 \section*{Acknowledgements}
F. Cuenca and C. Draper have received support from the Spanish Ministerio de Ciencia e Innovaci\'on,   grant PID2023-152673GB-I00, with FEDER funds,  and Junta de Andaluc\'{\i}a  grant  FQM-336.
T. Meyer     has received a University of Cape Town Science Faculty PhD Fellowship and the Harry Crossley Research Fellowship.

\bibliographystyle{plain}
\bibliography{ref} 
\Addresses
\end{document}